\documentclass[11pt,a4paper]{article}
\usepackage{fancyhdr}
\usepackage{amsmath,mathtools}
\usepackage{bm}
\usepackage{esint}
\usepackage{amsfonts}
\usepackage{amsthm}
\usepackage{amssymb}
\usepackage{amsbsy}
\usepackage{verbatim}
\usepackage{graphicx}
\usepackage{url}
\usepackage{mathrsfs}
\usepackage[hmargin = 1in,vmargin=.75in]{geometry}
\usepackage{color}
\usepackage{amstext}
\usepackage{stmaryrd}
\usepackage{enumerate}
\usepackage{algpseudocode}
\usepackage{algorithm}
\usepackage{pifont}
\usepackage{prettyref}
\usepackage{subcaption}

\font\msbm=msbm10

\numberwithin{equation}{section}

\theoremstyle{plain}
\newtheorem{theorem}{Theorem}[section]
\newtheorem{lemma}[theorem]{Lemma}
\newtheorem{corollary}[theorem]{Corollary}

\newtheorem{proposition}[theorem]{Proposition}

\newtheorem{open}[theorem]{Open Problem}
\newtheorem{remark}[theorem]{Remark}
\def\mathbb#1{\hbox{\msbm{#1}}}

\newcommand{\be}{\boldsymbol{e}}

\newcommand{\bg}{\boldsymbol{g}}

\newcommand{\bu}{\boldsymbol{u}}
\newcommand{\bv}{\boldsymbol{v}}
\newcommand{\bw}{\boldsymbol{w}}
\newcommand{\bx}{\boldsymbol{x}}
\newcommand{\by}{\boldsymbol{y}}
\newcommand{\bz}{\boldsymbol{z}}
\newcommand{\bone}{\boldsymbol{1}}

\newcommand{\btheta}{\boldsymbol{\theta}}

\newcommand{\BA}{\boldsymbol{A}}

\newcommand{\BE}{\boldsymbol{E}}

\newcommand{\BJ}{\boldsymbol{J}}
\newcommand{\BL}{\boldsymbol{L}}

\newcommand{\BQ}{\boldsymbol{Q}}
\newcommand{\BR}{\boldsymbol{R}}

\newcommand{\BU}{\boldsymbol{U}}
\newcommand{\BV}{\boldsymbol{V}}
\newcommand{\BW}{\boldsymbol{W}}
\newcommand{\BX}{\boldsymbol{X}}
\newcommand{\BY}{\boldsymbol{Y}}
\newcommand{\BZ}{\boldsymbol{Z}}

\newcommand{\BDelta}{\boldsymbol{\Delta}}

\newcommand{\bzero}{\boldsymbol{0}}

\newcommand{\pa}{\partial}

\newcommand{\I}{\boldsymbol{I}}
\newcommand{\RR}{\mathbb{R}}
\newcommand{\lag}{\langle}
\newcommand{\rag}{\rangle}

\newcommand{\Tr}{\text{Tr}}
\newcommand{\eps}{\epsilon}

\renewcommand{\Pr}{\mathbb{P}}
\newcommand*\diff{\mathop{}\!\mathrm{d}}

\DeclareMathOperator{\Real}{Re}

\DeclareMathOperator{\Var}{Var}
\DeclareMathOperator{\ddiag}{ddiag}
\DeclareMathOperator{\E}{\mathbb{E}}

\DeclareMathOperator{\diag}{diag}

\DeclareMathOperator{\Imag}{Im}

\newcommand{\mi}{\mathrm{i}}

\begin{document}
\title{\bf On the Landscape of Synchronization Networks: \\ A Perspective from Nonconvex Optimization}

\author{Shuyang Ling\thanks{Courant Institute of Mathematical Sciences and Center for Data Science, New York University, NY 10012.},~~Ruitu Xu\thanks{Courant Institute of Mathematical Sciences, New York University, NY 10012. },~~Afonso S. Bandeira$^{*}$\thanks{Afonso S. Bandeira was partially supported by NSF grants DMS-1712730 and DMS-1719545, and by a grant from the Sloan Foundation.}}
\maketitle

\begin{abstract}
Studying the landscape of nonconvex cost function is key towards a better understanding of optimization algorithms widely used in signal processing, statistics, and machine learning. Meanwhile, the famous Kuramoto model has been an important mathematical model to study the synchronization phenomena of coupled oscillators over various network topologies. 
In this paper, we bring together these two seemingly unrelated objects by investigating the optimization landscape of a nonlinear function $E(\boldsymbol{\theta}) = \frac{1}{2}\sum_{1\leq i,j\leq n}  a_{ij}(1-\cos(\theta_i - \theta_j))$ associated to an underlying network and exploring the relationship between the existence of local minima and network topology. This function arises naturally in Burer-Monteiro method applied to $\mathbb{Z}_2$ synchronization as well as matrix completion on the torus. Moreover, it corresponds to the energy function of the homogeneous Kuramoto model on complex networks for coupled oscillators. 
We prove the minimizer of the energy function is unique up to a global translation under deterministic dense graphs and Erd\H{o}s-R\'enyi random graphs with tools from optimization and random matrix theory. Consequently, the stable equilibrium of the corresponding homogeneous Kuramoto model is unique and the basin of attraction for the synchronous state of these coupled oscillators is the whole phase space minus a set of measure zero. In addition, our results address when the Burer-Monteiro method recovers the ground truth exactly from highly incomplete observations in $\mathbb{Z}_2$ synchronization and shed light on the robustness of nonconvex optimization algorithms against certain types of so-called monotone adversaries. Numerical simulations are performed to illustrate our results.

\end{abstract}

\section{Introduction}
Nonconvex optimization plays a crucial role in the mathematics of signal processing~\cite{BBV16,ChenC18,LZT18,SunQW15,SQW17,SunQW16,ZKW18}, statistics~\cite{CL17,MPW16}, and machine learning~\cite{CG18,GeLM16,GM17,SJL17,SunL16}, and has attracted substantial attention in the past few years. Unlike convex optimization, nonconvex objective functions often exhibit complex geometric structures, making it essential to analyze the optimization landscape and to characterize the location and number of local optima and critical points respectively.

This paper concerns the optimization landscape of the following nonconvex function, 
\begin{equation}\label{def:E}
E(\btheta) := \frac{1}{2}\sum_{1\leq i,j\leq n}  a_{ij}(1-\cos(\theta_i - \theta_j))
\end{equation}
where $\BA = (a_{ij})_{1\leq i,j\leq n}$ with $a_{ij}\geq 0$ is the weight matrix of an undirected graph. As a special case, $\BA$ can be an adjacency matrix. 
We want to address the following key question:
\begin{quote}
{\bf Key question: }{\em What is the relationship between the existence of local minima and the topology of the network?}
\end{quote}

One may wonder what motivates us to study the optimization landscape of $E(\btheta)$ and why it is an important and valuable topic. 
Surprisingly, this seemingly simple function appears in numerous literature in mathematics, physics, engineering, and computer science. Despite the tremendous efforts, many mysteries behind this simple function still remain unsolved. 
We briefly introduce how this function emerges in various contexts in the next sections.

Before we move on, we first have a preliminary analysis of $E(\btheta).$
Note that it suffices to consider $E(\btheta)$ over $[0, 2\pi)^n$ due to the periodicity. A direct computation implies that $\btheta_0 \coloneqq \bzero$ modulo a global translation, i.e., $\theta_i =\theta_j$ for all $i\neq j$, is a global minimizer of $E(\btheta)$, independent of the underlying graph $\BA$ (from now on, we use $\BA$ to denote the weight (adjacency) matrix as well as the graph it represents). However, it is unclear whether there exist~\emph{spurious local minima} besides the global minimum achieved at $\btheta=\btheta_0.$ In particular, we say the optimization landscape of $E(\btheta)$ is~\emph{benign} if there is no spurious local minimum at all. 

\subsection{Group synchronization, matrix completion, and monotone adversaries}
Suppose there are $n$ group elements $\{g_i\}_{i=1}^n$ of a group $G$. Instead of observing these elements directly, only a subset of their noisy offsets $g_ig_j^{-1}$ is available. The goal is to recover all the group elements up to a global phase from $\{g_ig_j^{-1}\}_{\{(i,j)\in\mathcal{E}\}}$ where $\mathcal{E}$ denotes the index set of available observations. This is called the~\emph{group synchronization} problem, see~\cite{AbbeBBS14,AMMS17,BBS17,BBV16,B16,ChenC18,PWBM16,ZB18} for more details. 

The simplest example of group synchronization is $\mathbb{Z}_2$ synchronization~\cite{AbbeBBS14} where $g_i$ is assumed to belong to $\{\pm 1\}$ and thus the offsets reduce to $g_ig_j^{-1} = g_ig_j.$ 
More precisely, we have $\BY$ as the observed data, which is the entrywise product of a binary symmetric matrix $\BA\in\{0,1\}^{n\times n}$ and the perturbed ground truth $\bg\bg^{\top} + \BW$, i.e.,
\[
\BY = \BA\circ (\bg\bg^{\top} + \BW)
\]
where $\BW$ is the noise matrix and $``\circ"$ denotes the entrywise multiplication. Without the group constraints on $\{g_i\}_{i=1}^n$, it is exactly a matrix completion problem~\cite{CandesR09,GeLM16,Recht11,SunL16} which originally comes from the {\em Netflix Prize} problem.

One possible formulation is to maximize the following combinatorial quadratic form
\[
\max_{\bz\in\{\pm 1\}^n} \bz^{\top}\BY\bz
\]
and then use the maximizer as an estimator of $\bg$ from its partial noisy measurements. In fact, the program above is equivalent to the least squares loss function after a simple transformation. However,  solving this program is NP-hard in general. One may consider semidefinite programming (SDP) relaxation as an alternative by trying to recover the rank-1 matrix $\bg\bg^{\top}$ instead of individual group elements:
\[
\max~\Tr( \BY\BZ) \quad \text{s.t. } \BZ\succeq 0, ~Z_{ii} = 1,~\forall\, 1\leq i\leq n.
\]
Note the ground truth $\bg\bg^{\top}$ satisfies the two constraints above. Under mild conditions, this SDP relaxation is proven to recover $\bg\bg^{\top}$ exactly~\cite{AbbeBBS14}. Despite the effectiveness of SDP relaxation~\cite{RFP10}, it tends to suffer from a high computational complexity. 
Hence, numerous efforts have been taken in order to find efficient and robust gradient-descent-based methods as alternatives to solve the SDP arising from the relaxation of the low-rank matrix recovery problem.

The approach proposed by Burer and Monteiro in~\cite{BurerM03,BurerM05} may be arguably
one of the most successful methods to tackle these otherwise difficult problems by conveniently keeping $\BZ$ in a factorized form and taking full advantage of the low-rank property.
By letting $\BZ = \BQ\BQ^{\top}$ where $\BQ\in\RR^{n\times q}$ with $q<n$,  we have
\begin{equation}\label{prog:bmq}
\max~\Tr(\BY\BQ \BQ^{\top}) \quad \text{s.t. } (\BQ\BQ^{\top})_{ii} = 1,~\forall\, 1\leq i\leq n.
\end{equation}
In other words, each row of $\BQ$ needs to be normalized. These constraints yield a search space consisting of a Riemannian manifold, and Riemannian gradient descent is used to solve this program~\cite{AMS09,BBV16}. 

In particular, $E(\btheta)$ is equivalent to Burer-Monteiro method~\cite{BBV16,BurerM03,BurerM05} applied to $\mathbb{Z}_2$ synchronization on networks when $q=2$. Without loss of generality, we assume $\bg = \bone_n$ by considering $\diag(\bg)\BY\diag(\bg)$ instead of $\BY$.
The $i$-th row of $\BQ$ as $[\cos\theta_i,~\sin\theta_i]$ gives $(\BQ\BQ^{\top})_{ij} = \cos(\theta_i - \theta_j)$. Thus we have an equivalent form as 
\begin{equation}\label{prog:E1}
\min_{\btheta\in\RR^n}~~\sum_{i,j} a_{ij}(1 + w_{ij}) (1-\cos(\theta_i - \theta_j))
\end{equation}
which is in the exact form of $E(\btheta)$.

Note that the Burer-Monteiro method yields a nonconvex cost function. This fact immediately brings up a number of questions: when does this Burer-Monteiro approach recover the ground truth? Does~\eqref{prog:E1} have spurious local minima? How does the existence of spurious local minima depend on the topology of the network? 
All those problems reduce to the analysis of the optimization landscape of $E(\btheta)$ in~\eqref{def:E} under different network topologies. 

The answers to these aforementioned questions are closely related to the robustness of algorithms against~\emph{monotone adversaries}, an increasingly important topic in algorithm design in theoretic computer science. 
For simplicity, we assume no noise in the measurements, i.e., $w_{ij}=0.$ Suppose we collect more and more noiseless observations. It is not hard to imagine that more noiseless data help us to recover the underlying signal $\bg$. Moreover, it can be easily proven that the SDP relaxation works perfectly under this scenario. On the other hand, the Burer-Monteiro approach exhibits satisfactory performance in practice and enjoys much higher efficiency than the SDP relaxation does. However, it is unclear whether the Burer-Monteiro approach would inherit the robustness to such seemingly helpful ``noise" from the SDP relaxation. This type of ``noise" is often referred to as a \emph{monotone adversary}. 
Note that more observations mean adding more edges to the current network and thus we propose the following question:
\begin{quote}
{\em Does adding more edges to the network improve the optimization landscape or worsen it by creating spurious local optima?}
\end{quote}
We defer a more detailed discussion regarding monotone adversaries to Section~\ref{ss:mono}.

\subsection{The Kuramoto model and synchronization on complex networks }
The synchronization problem of coupled oscillators has a long history. It was originally proposed and studied by Christiaan Huygens in 1665 when investigating the behavior of two pendulum clocks mounted side by side on the same support. He observed that the two oscillators would always end up swinging in exactly opposite directions, independent of their respective motion~\cite{RONA16,ADKM08}. 

Since then, the synchronization of coupled oscillators has fascinated the scientific community, and remarkable progress has been made regarding this topic.
One of the breakthroughs was made by Kuramoto~\cite{Kura75,Kura84} who came up with a mathematically tractable model to describe the synchronization behavior for a large set of coupled oscillators. The general Kuramoto model assumes $n$ oscillators which interact with one another based on sinusoidal coupling, i.e., 
\begin{equation}\label{eq:kuramoto}
\frac{\diff \theta_i}{\diff t} = \omega_i - \sum_{j=1}^n a_{ij}\sin(\theta_i - \theta_j), \quad 1\leq i\leq n
\end{equation}
where each $\theta_i\in[0,2\pi)$ is a function of time $t$. Here $\omega_i$ refers to the natural frequency of the $i$-th oscillator, and $a_{ij}$ stands for the strength of the mutual coupling. 
The original work by Kuramoto~\cite{Kura75} required the weights $\{a_{ij}\}$ to be identical and showed the oscillators $\{\theta_i\}_{i=1}^n$ would synchronize if the mutual interactions are stronger than the effect of natural frequencies $\{\omega_i\}$ as discussed in~\cite{Kura75,Kura84}. In particular, the local asymptotic stability of~\eqref{eq:kuramoto} with general network topologies and natural frequencies is first studied in~\cite{JMB04}.
It is also worth noting that the Kuramoto model appears in quite many applications such as electric power networks, neuroscience, chemical oscillations, spin glasses, see~\cite{ABVRS05,ADKM08,DB14,DCB13,RPJK16} and the references therein for more details.

\begin{figure}[h!]
\centering
\includegraphics[width=50mm]{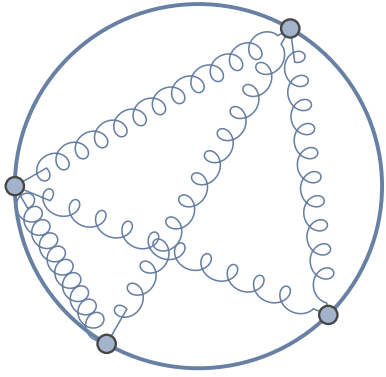}
\caption{Illustration of coupled oscillators.}
\label{fig:osc}
\end{figure}

The Kuramoto model was later studied on complex networks~\cite{JMB04,DCB13,ABVRS05,RPJK16} where $a_{ij} \geq 0$ describes how much the $i$-th and the $j$-th oscillator are coupled. In particular, if $\BA$ is an adjacency matrix, it means two oscillators $i$ and $j$ are coupled if $a_{ij}=1$, see Figure~\ref{fig:osc}.

Note that the synchronous state is highly related to the stable equilibrium of continuous time dynamical system~\eqref{eq:kuramoto}. On the other hand,~\cite{DCB13} points out the coupled oscillators model is a system of particles which tends to minimize the following energy function,
\[
E_{\omega}(\btheta) = \sum_{i,j}a_{ij} (1-\cos(\theta_i - \theta_j)) - \sum_{i=1}^n \omega_i \cdot \theta_i
\]
where the two terms above model the particle interaction and the external force respectively. If we consider a standard Cauchy problem associated to $E_{\omega}(\btheta)$ by setting the velocity of each oscillator to the gradient of $-E_{\omega}(\btheta)$, i.e.,
\[
\frac{\diff \theta_i(t)}{\diff t} = -\frac{\pa E_{\omega}(\btheta)}{\pa \theta_i}, \quad 1\leq i\leq n, \quad \theta_i = \theta_i(0)
\]
or equivalently,
\begin{equation}\label{def:cauchy}
\frac{\diff \btheta(t)}{\diff t} = -\nabla E_{\omega}(\btheta),\quad \btheta = \btheta(0)
\end{equation}
then the evolution is always moving in the direction where the energy $E_{\omega}(\btheta)$ decreases maximally, which is referred to as gradient flows.  Discretizing~\eqref{def:cauchy} by an Euler scheme corresponds to gradient descent applied to the energy function $E_{\omega}(\theta).$

 It is natural to ask the behavior of $\btheta$ as time evolves. In fact each stable equilibrium corresponds to a local minimizer of energy function $E_{\omega}(\btheta)$ whose existence is the main focus of this work. From this point on, we consider a simplified version, the homogeneous Kuramoto model~\cite{Kura75,Strogatz00,Taylor12}, where the natural frequencies $\{\omega_i\}_{i=1}^n$ are 0. Under this assumption, $E_{\omega}(\btheta)$ reduces to $E(\btheta)$ and thus it suffices to understand the local minima of $E(\btheta)$ in order to characterize the synchronization of the homogeneous Kuramoto model on networks.


\subsection{Related work and our contributions}
The landscape of nonconvex loss functions has drawn extensive attention recently. Examples range from signal processing to statistics and machine learning, including phase retrieval~\cite{SunQW15,SunQW16}, matrix completion~\cite{GeLM16}, blind deconvolution~\cite{ZKW18}, community detection and synchronization~\cite{BBV16}, kernel PCA~\cite{CL17}, low-rank matrix sensing problem~\cite{LZT18}, tensor decomposition~\cite{GM17}, dictionary learning~\cite{SQW17}, and neural networks~\cite{SJL17}. Among all those existing works, ours shares more commonalities with group synchronization and matrix completion. 

Synchronization problems, which aim to recover group elements $g_i$ from its relative alignment $g_ig_j^{-1}$ (or noisy alignment), are considered on different groups such as
phase synchronization~\cite{BBS17,ZB18} (corresponding to the group $U(1)$), $\mathbb{Z}_2$ synchronization~\cite{AbbeBBS14,BBV16}, joint alignment from pairwise differences~\cite{ChenC18} (corresponding to the cyclic group, i.e., $\mathbb{Z}/n\mathbb{Z}$), and general group synchronization problems~\cite{PWBM16,AMMS17}. 
Many algorithms, some based on convex and nonconvex optimization, are proposed to tackle these problems and proven to work well on group elements retrieval. In these aforementioned works, we are mostly influenced by $\mathbb{Z}_2$ and phase synchronization~\cite{BBS17,ZB18,AbbeBBS14,BBV16,ChenC18}, especially~\cite{BBV16}.

The convex relaxation approach for $\mathbb{Z}_2$ and phase synchronization is initially studied in~\cite{BBS17,AbbeBBS14}. The proposed method is shown to be tight, i.e., recovery group elements exactly under mild conditions. Later,~\cite{B16,ChenC18,ZB18} show that one can retrieve group elements via a two-step nonconvex optimization procedure, i.e., a carefully chosen initialization followed by (projected) power iteration. Compared with convex relaxation, iterative methods based on gradient descent and power methods achieve much higher efficiency while still enjoying rigorous theoretic guarantees. However, the discussion on the landscape of corresponding cost function is not covered~\cite{ZB18,ChenC18}. 
In~\cite{BBV16}, the authors studied the Burer-Monteiro method applied to $\mathbb{Z}_2$ synchronization and obtained an objective function in the form of $E(\btheta)$ (modulo a proper transformation). The resulting function is shown to enjoy a benign optimization landscape by taking advantage of the randomness in the weight matrix. The work provides a strategy for analyzing the landscape of energy function via an effective perturbation argument. Unlike~\cite{BBV16}, our work concerns the properties of landscapes of the energy function associated to an unweighted network and how the landscape depends on the network topology, which is a vastly different setting.

Our work is also closely related to matrix completion~\cite{GeLM16,CandesR09,SunL16} where one wants to recover a low-rank matrix from partial observations. The landscape of least squares loss function associated with matrix completion is proven to be benign in~\cite{GeLM16}; fast and efficient nonconvex algorithms are developed in~\cite{KOM09,SunL16} to solve matrix completion problem. The main distinction of our work lies in the fact that the ground truth signals have entrywise unit modulus. In fact, the additional unit modulus constraints lead to a more complicated landscape since they potentially create many critical points and local optima, thus making the analysis different. 

\vskip0.25cm
Another direction of relevant research comes from dynamical systems and synchronization on complex networks~\cite{Strogatz01}. One prominent example is the Kuramoto model~\cite{Kura75} which models the collective behavior of a large population of oscillators which interact with one another via sinusoidal coupling~\cite{DB14,DCB13,RPJK16,Strogatz00,WileySG06,JMB04}. 
We gladly acknowledge being influenced by the excellent reviews in this field~\cite{DCB13,DB14,RPJK16}. For the Kuramoto model, the core problem is to find and analyze the equilibria which yields solving a nonlinear system of equations~\cite{MDDH15} and the stable equilibria correspond to the local minima of the energy function. 
The work~\cite{Strogatz00} introduces a class of illuminating examples of networks called Wiley-Strogatz-Girvan networks and their corresponding energy functions of the coupled oscillators yield spurious local minima. These examples indicate that there may exist local minima even if the graph is well connected. The work by Taylor~\cite{Taylor12} on homogeneous Kuramoto model suggests exploring the existence of local minima for deterministic dense graphs and plays a role in guiding our theoretic analysis. Several recent works~\cite{CMM18,Ichi04,LLYMG16,MT15} have studied the Kuramoto model on random graphs via mean field analysis; they do not overlap with the current work since we consider the landscape of energy function with finite $n$ instead of its continuum limit and asymptotic behavior when $n\rightarrow \infty$. Another distinct feature of our work is on the analysis of global landscape instead of the local optimization landscape restricted on the cohesive phases, i.e., $\{\btheta: |\theta_i - \theta_j| < \pi/2\}$, as seen in~\cite{DB14,DCB13}.

\vskip0.25cm
The contribution of this work is multifold. We give a precise characterization of the connection between the function landscape and the associated network topology. We consider two important types of graphs: dense deterministic graphs and Erd\H{o}s-R\'enyi random graphs, and prove the resulting function landscape enjoys no other local minima except the synchronized state $\btheta = \btheta_0$. In other words, we prove that the homogeneous Kuramoto model has only one stable synchronized solution and the basin of synchronization is the whole phase space with high probability if the underlying graph is a \emph{finite} Erd\H{o}s-R\'enyi graph. 
We also prove that if each node has at least $(3-\sqrt{2})(n-1)/2 \approx 0.7929(n-1)$ neighbors, no spurious local minima exist, which improves the state-of-the-art bound $\mu\geq0.9375(n-1)$ in~\cite{Taylor12} by Taylor.
Moreover, the theoretic analysis provides explicit examples under which the Burer-Monteiro method is not robust against monotone adversaries. The techniques involved in this manuscript share very little similarities with those used in the existing literature of the Kuramoto model on complex networks. Thus our work likely offers new inspirations to solve the Kuramoto models on other types of random graphs.
In conclusion, our work not only provides solutions to a few important questions about synchronization problem on networks but also motivates a series of open problems.

\vskip0.25cm

\subsection{Organization of our paper}
This paper is organized as follows. In Section~\ref{s:prelim} we go through the basics of $E(\btheta)$ and introduce a few examples. Our main theorems and open problems are presented in Section~\ref{s:main}. Section~\ref{s:numerics} is devoted to numerical experiments that complement the theoretical analysis.  Finally, the proofs of our results can be found in Section~\ref{s:proofs}.

\subsection{Notation}
We introduce notations which will be used throughout the paper. Matrices are denoted in boldface such as
$\BZ$; vectors are denoted by boldface lower case letters, e.g.~$\bz.$
The individual entries of a matrix or a vector are denoted in normal font such as $z_{ij}$ or
$z_i.$
For any matrix $\BZ$, $\|\BZ\|$
denotes its operator norm, and $\|\BZ\|_F$ denotes its the Frobenius norm. For any vector $\bz$, $\|\bz\|$ denotes its Euclidean norm; $\|\bz\|_{\infty}$ stands for the $\ell_{\infty}$ norm. For both
matrices and vectors, $\BZ^{\top}$ and $\bz^{\top}$ stand for the transpose of $\BZ$ and $\bz$ respectively. We equip the matrix space $\RR^{K\times N}$ with the inner
product defined by $\lag \BU, \BV\rag \coloneqq\Tr(\BU^{\top}\BV).$ A special case is the inner product of two vectors, i.e.,
$\lag \bu, \bv\rag = \Tr(\bu^{\top}\bv) = \bu^{\top}\bv.$ For a given vector $\bv$, $\diag(\bv)$ represents the diagonal
matrix whose diagonal entries are given by the vector $\bv$; and $\ddiag(\BZ)$ denotes a diagonal matrix whose diagonal entries are the same as those of $\BZ$.
$\I_n$ and $\bone_n$ always denote the $n\times n$ identity matrix and a column vector of ``$1$" in $\RR^n$ respectively; $\BJ_n$ represents $\bone_n\bone_n^\top$; $\be_i\in\RR^n$ denotes a vector with a 1 in the $i$-th coordinate and 0's elsewhere.
For any $\BU$ and $\BV$ in $\RR^{K\times N}$, their Hadamard product is denoted as $\BU\circ\BV$. For any $\BZ\in\RR^{N\times N}$, $\BZ\succeq 0$ if $\BZ$ is symmetric and positive semidefinite.

\section{Preliminaries}\label{s:prelim}
This section is devoted to providing basic properties of $E(\btheta)$ and some concrete examples. The examples illustrate the connection between the energy landscape and network topology, motivating our theoretical and numerical investigations. Moreover, they address the robustness issue of the Burer-Monteiro approach against monotone adversary.

\subsection{Problem setup}
We start with $E(\btheta)$ and write it into a more compact form. 
Let 
\begin{equation}\label{def:Q}
\BQ := [\bx~\by]\in\RR^{n\times 2}
\end{equation}
 where $\bx \coloneqq [\cos \theta_1, \cdots, \cos\theta_n]^{\top}\in\RR^n$ and $\by \coloneqq [\sin\theta_1, \cdots, \sin\theta_n]^{\top}\in\RR^n$. Then we have $(\BQ\BQ^{\top})_{ij} = x_ix_j + y_iy_j= \cos(\theta_i - \theta_j)$ where $x_i$ and $y_i$ are the $i$-th entry of $\bx$ and $\by$ respectively.
Hence $E(\btheta)$ can be recast in terms of $\BQ$, 
\[
E(\btheta) = \frac{1}{2} \lag \BA, \BJ_n  - \BQ\BQ^{\top}\rag.
\]

To study the behavior of local minima, it is essential to obtain the first and second order necessary condition. 
Both conditions are easy to obtain via direct computation:
\begin{equation}\label{cond:1st}
(\nabla E(\btheta))_i = \sum_{j=1}^n a_{ij}\sin(\theta_i - \theta_j) = 0.
\end{equation}
Thanks to a simple trigonometric identity $\sin(\theta_i - \theta_j) = \sin\theta_i\cos \theta_j - \cos\theta_i\sin\theta_j$, we have 
\[
\BA\bx \circ \by = \BA\by\circ \bx.
\]

Now we proceed to compute the Hessian matrix of $E(\btheta)$,
\begin{equation}\label{eq:2nd}
(\nabla^2 E(\btheta))_{ij} =
\begin{dcases}  
\frac{\pa^2 E}{\pa \theta_i\pa\theta_j} = -a_{ij}\cos(\theta_i - \theta_j), & \forall i\neq j, \\
\frac{\pa^2 E}{\pa \theta_i^2} = \sum_{j\neq i}a_{ij}\cos(\theta_i - \theta_j), & \forall i = j.
\end{dcases}
\end{equation}

In fact, the first and second order derivatives above are exactly the Riemannian gradient and Hessian of~\eqref{prog:bmq} with a proper trigonometric transformation since one can think of~\eqref{prog:bmq} as an optimization problem on a $n$-dimensional torus. Interested readers may refer to ~\cite{journee2010low,AMS09,BBV16} for more details about how to solve optimization problems on manifold.

The Hessian matrix $\nabla^2 E(\btheta)$ can be regarded as the graph Laplacian associated to the weight matrix $\{ a_{ij}(\cos(\theta_i - \theta_j))\}_{1\leq i,j\leq n}$ if we allow negative weights. As a result, the constant vector $\bone_n$ is in the null space of $\nabla^2 E(\btheta).$
In particular, when $\btheta$ is a local minimizer of $E(\btheta)$, we have $\nabla^2 E(\btheta)\succeq 0$. This necessary condition has the following equivalent form,
\begin{equation}\label{cond:2nd} 
\nabla^2 E(\btheta) = \ddiag(\BA\BQ\BQ^{\top}) - \BA\circ \BQ\BQ^{\top} \succeq 0.
\end{equation}

Following from~\eqref{cond:2nd}, we know that if $\{\theta_i\}_{i=1}^n$ simultaneously stay inside a single quadrant, then $\nabla^2 E(\btheta)\succeq 0$ and $E(\btheta)$ is convex.
This is because $\nabla^2 E(\btheta)$ is the graph Laplacian associated to the weight given by $a_{ij}\cos(\theta_i - \theta_j) \geq 0$ if $|\theta_i - \theta_j| < \pi/2$. In particular, if the graph is connected, $E(\btheta)$ is strictly convex within $\{\btheta: |\theta_i - \theta_j| < \pi/2\}.$ 
 Therefore, local convexity holds in a neighborhood around the global minimizer $\btheta_0 = \bzero$. 
 
However, the global landscape is more complex. In fact, $E(\btheta)$ has at least $2^n$ distinct critical points in the form of $\{ \btheta\in\RR^n : \theta_i = 0 \text{ or } \pi\}$ up to a global rotation.  Moreover, other critical points and local optima may coexist while not be obtained analytically. This adds difficulties to the landscape analysis of $E(\btheta)$.

\subsection{Examples}
We introduce a few examples to illustrate the complex relationship between the existence of spurious local minima of $E(\btheta)$ and the network represented by $\BA$. First, it is natural to ask whether there are no spurious local minima at all if $\BA$ is well connected. Note that the Hessian matrix at $\btheta_0 = \bzero$ satisfies
\[
\nabla^2 E(\btheta_0) = \ddiag(\BA \BJ_n) - \BA = \diag(\BA\bone_n) - \BA
\]
which corresponds exactly to the graph Laplacian associated to the adjacency matrix $\BA.$ When adding more edges to the graph, i.e., putting more nonzero entries into $\BA$, the second smallest eigenvalue of $\nabla^2E(\btheta)$ would increase and thus strengthen the local convexity at $\btheta_0$. However, there are examples indicating that the local curvature at $\btheta_0$ {\em fails} to determine the properties of the global landscape we are interested in. We defer the technical details behind these examples to Section~\ref{ss:proofex}.

\paragraph{The path graph of $n$ nodes.}
We start with a simple example where the graph is a path of $n$ vertices, see Figure~\ref{fig:graph1}. For such a path graph, we claim there are exactly $2^n$ critical points represented by $\{\btheta \in\RR^n: \theta_i = 0 \text{ or } \pi \}$ but the
only local minimizer is $\btheta = \btheta_0$. More precisely, the adjacency matrix of a path satisfies 
\[
a_{ij} = 
\begin{cases}
1, & \text{if }|j-i|=1,\\
0, &  \text{otherwise}.
\end{cases}
\]
By substituting $a_{ij}$ into~\eqref{cond:1st}, solving for critical points, and checking the second order necessary condition, we arrive at our claim. In fact, if the graph is a connected tree (i.e., a connected graph without loop), one can draw a similar conclusion. 

\begin{figure}[h!]
\centering
\includegraphics[width=45mm]{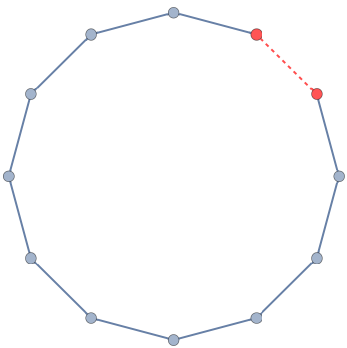}
\caption{A path graph with $n$ nodes and an $n$-cycle. While the energy function associated to a path graph has the unique minimum $\btheta_0$, it has at least one spurious local minimum if the underlying network is an $n$-cycle. }
\label{fig:graph1}
\end{figure}

\paragraph{Wiley-Strogatz-Girvan networks.}
The Wiley-Strogatz-Girvan networks~\cite{WileySG06} are constructed via linking each node on a cycle with its $k$-nearest neighbors, see Figure~\ref{fig:WSG}. Obviously, the larger $k$ is, the better the graph is connected. On the other hand, for such a graph, at least one local minimum besides $\btheta_0$ can be constructed explicitly if $k \leq 0.34n$, i.e., each node has at most $0.68n$ neighbors.
More precisely, we can verify that $\btheta = 2\pi n^{-1} [0, 1, \cdots,  n-1]^{\top}$ is a local minimizer (a.k.a. the uniformly twisted states, see~\cite{WileySG06}), and we provide a short derivation via the Discrete Fourier Transform in Section~\ref{ss:proofex}.
\begin{figure}[h!]
\centering
\begin{minipage}{0.48\textwidth}
\includegraphics[width=60mm]{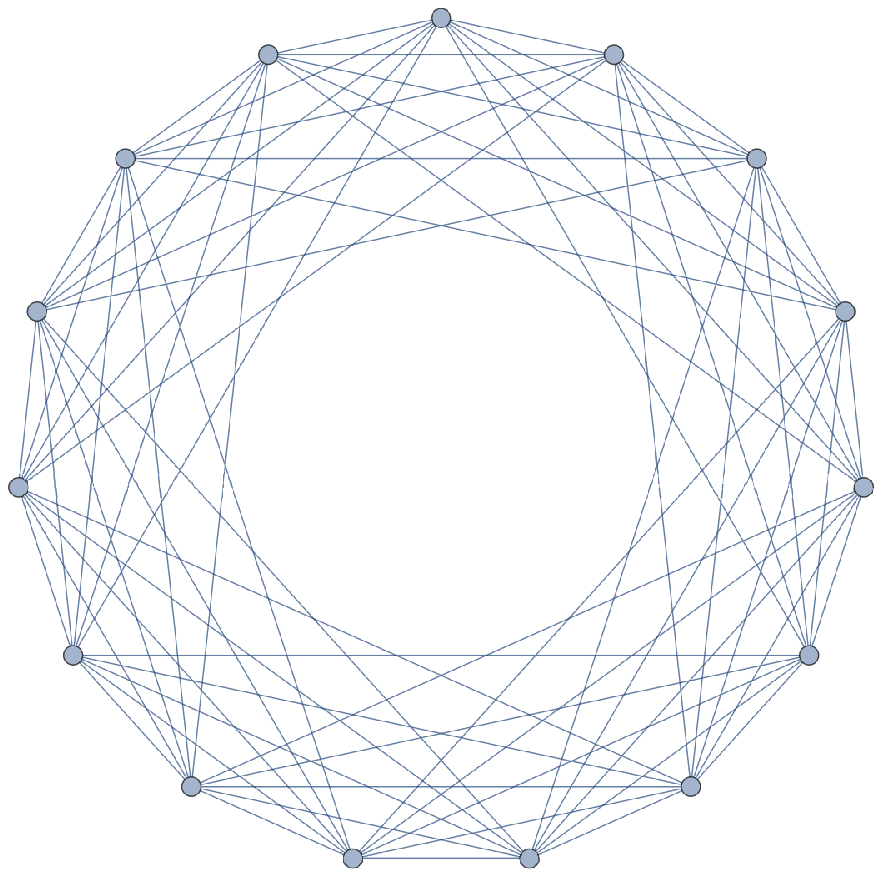}
\end{minipage}
\hfill
\begin{minipage}{0.48\textwidth}
\includegraphics[width=60mm]{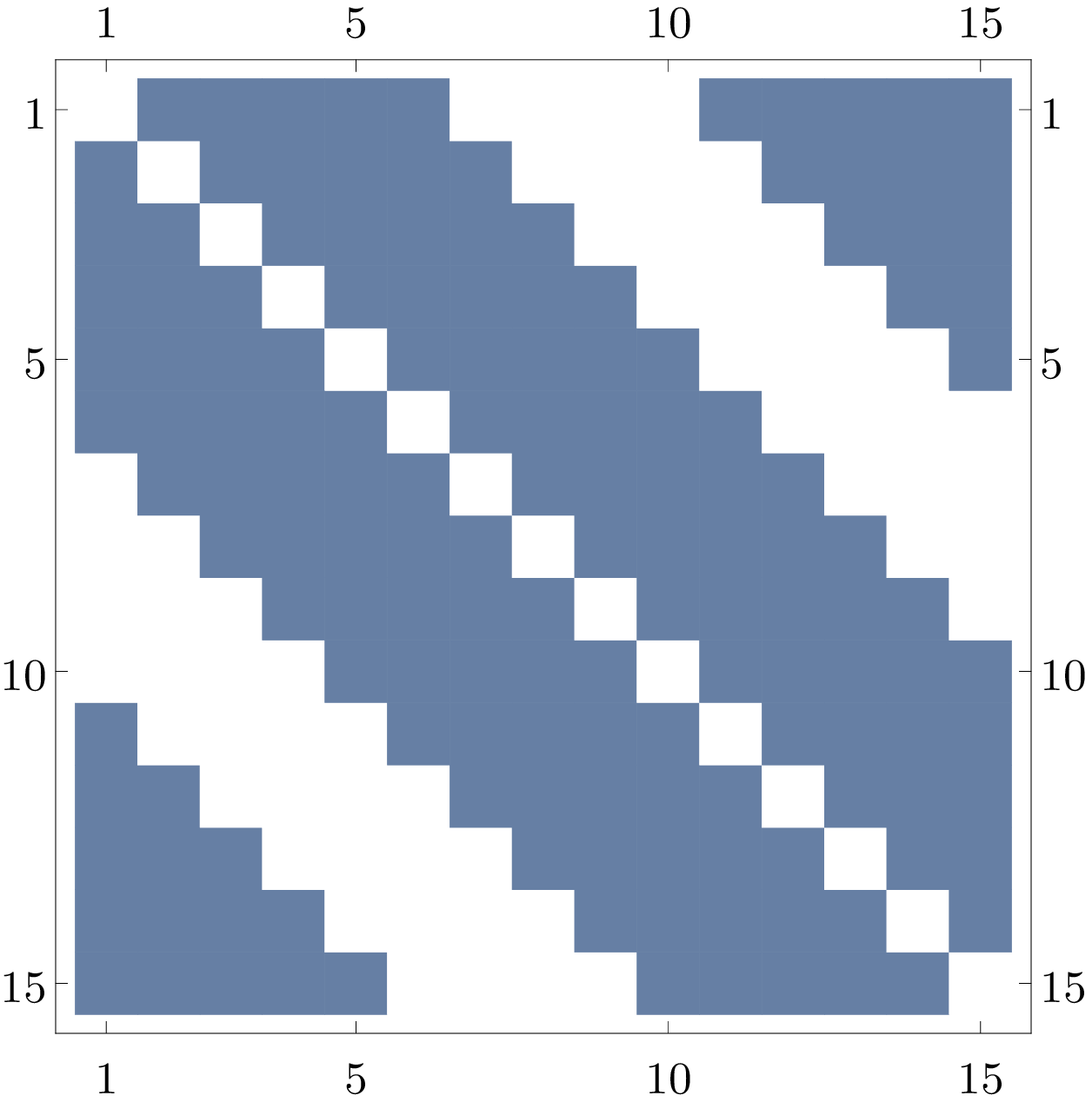}
\end{minipage}
\caption{Wiley-Strogatz-Girvan networks~\cite{WileySG06} and its adjacency matrix.}
\label{fig:WSG}
\end{figure}

\begin{figure}[h!]
\centering
\includegraphics[width=80mm]{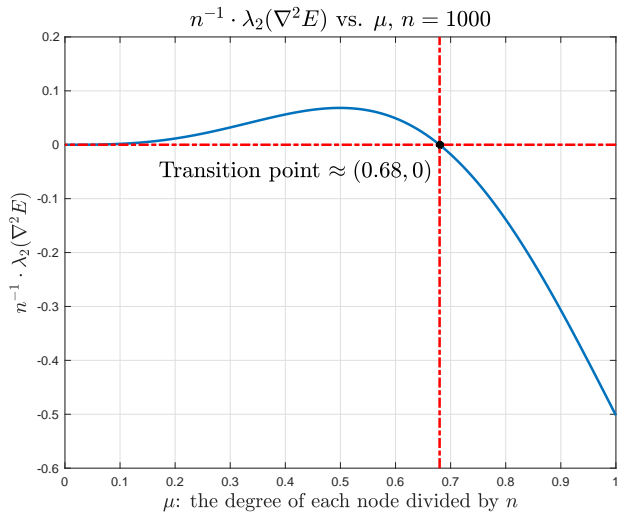}
\caption{The rescaled second smallest eigenvalue of Hessian matrix $(n-1)^{-1}\lambda_2(\nabla^2 E(\btheta))$ with respect to $\mu$ at the twisted state for Wiley-Strogatz-Girvan networks.}
\label{fig:lambda2}
\end{figure}

\vskip0.25cm
In conclusion, one can always find an explicit graph whose nodes are of degree as large as $0.68n$ such that at least one spurious local minimum exists. It is worth noting that if the underlying graph is complete, there is no local minimum~\cite{Taylor12}. Therefore, we want to find out if there exists a critical value for the degree which prevents the existence of spurious local minima.

\subsection{Monotone adversary}\label{ss:mono}
Now we return to the Burer-Monteiro approach when applied to $\mathbb{Z}_2$ synchronization and look into its robustness against the~\emph{monotone adversary}. Here the monotone adversary means adding noise that seemingly helps us to solve the problem by providing more information. The motivation behind this is to design efficient and reliable algorithms to tackle difficult optimization problems~\emph{without} implicitly taking advantage of the statistical properties of noise. Therefore, the study of monotone adversaries becomes an active research topic in algorithm designs in theoretic computer science~\cite{FK01} and has been discussed in several contexts such as community detection~\cite{MPW16} and matrix completion~\cite{CG18}. 

For $\mathbb{Z}_2$ synchronization, let us investigate the monotone adversary for a simplified model in~\eqref{prog:E1} with $\BW=0$ and study the landscape of the corresponding objective function associated to a given graph $\BA$. In this case, the objective function is exactly equal to $E(\btheta)$ and the monotone adversary is equivalent to adding more edges to the network. In other words, we have more observed data by doing that, and it should be more helpful for the retrieval of underlying information. Note that SDP relaxation is robust to this type of adversary as more information is provided. However, it is unclear whether nonconvex optimization methods such as Burer-Monteiro approach would enjoy this type of robustness property. 

Unfortunately, this is not the case in general as we have already seen explicit examples where adding more edges may {\em worsen} the landscape of $E(\btheta)$ by introducing undesired spurious local minima. Let $\BA$ be the adjacency matrix of the  $n$-path graph, and there is no local minimum at all. Suppose we add one more edge to make it into an $n$-cycle graph. The resulting graph becomes an example of the Wiley-Strogatz-Girvan networks and it is known that the twisted state leads to a spurious local minimum.

\section{Main results and open problems}\label{s:main}
The connection between the optimization landscape and general network topology is far from being well understood. 
From now on, we consider two important types of graphs and the optimization landscape of their corresponding energy function $E(\btheta)$. 
\begin{enumerate}[1.]
\item {\bf Deterministic dense graphs.}~~~Concrete examples from the Wiley-Strogatz-Girvan networks are constructed showing the existence of spurious local minima even if the graph is highly connected. On the other hand, we note that if the underlying graph is complete, i.e., $a_{ij} = 1$, no spurious local minimum exists. Therefore, one may immediately wonder how the optimization landscape looks like if the graph is sufficiently dense, i.e., each vertex has at least $\mu (n-1)$ neighbors for $\mu$ close to 1. We refer to this as the worst case scenario. 

\item  {\bf Erd\H{o}s-R\'enyi random graphs.}~~~While the deterministic dense graphs correspond to the worst case analysis, is it possible that the landscape of $E(\btheta)$ is benign on average? For this purpose, we consider the underlying graph as Erd\H{o}s-R\'enyi random graphs $G(n,p)$ whose adjacency matrix $\BA\in\RR^{n\times n}$ satisfies $a_{ij} = a_{ji}$ and 
\begin{equation}\label{eq:ER}
a_{ij} = 
\begin{cases}
1, & \text{ with probability } p, \\
0, & \text{ with probability } 1 - p,
\end{cases}
\end{equation}
for certain constant $p\in[0,1]$. This is called the average case scenario. 
\end{enumerate}

The first theorem states that if the graph is sufficiently dense, then the corresponding energy function has a benign optimization landscape, i.e., all local minima are global and attained at $\btheta = \btheta_0.$
\begin{theorem}\label{thm:main1}
Given a graph with $n$ vertices. Suppose the degree of each vertex is at least $\mu (n-1)$ with 
\[
\mu \geq \frac{3-\sqrt{2}}{2}, 
\]
then the only local minimum of the energy function $E(\btheta)$ is attained when
\[
\btheta = \btheta_0,
\]
which is unique modulo a global translation.
\end{theorem}
The result above significantly improves on a previous result by Taylor~\cite{Taylor12} where the optimization landscape is proven benign if $\mu \geq 0.9395.$ Combined with the existing facts from the Wiley-Strogatz-Girvan networks~\cite{WileySG06}, we have a more complete characterization regarding the existence of local minima with respect to the degree of each vertex: the smallest possible $\mu$ is between $0.6809$ and $(3-\sqrt{2})/2\approx 0.7929$ such that no spurious local minima exists if all nodes have at least $\mu (n-1)$ neighbors. Based on all the results obtained so far, we propose the following open problem:
\begin{open}
Does there exist a critical constant $\mu_c$ such that
\begin{itemize}
\item if $\mu < \mu_c$, there are other local minima besides $\btheta = \btheta_0$,
\item if $\mu > \mu_c$, the only local minimum is achieved at $\btheta=\btheta_0$,~(modulo a global translation)
\end{itemize}
where each node has at least $\mu(n-1)$ neighbors?
\end{open}

While the deterministic dense graphs yield the worst case analysis, we present an average case result with the help of randomness. 
\begin{theorem}\label{thm:main2}
Given an Erd\H{o}s-R\'enyi graph $G(n,p)$. Suppose $p = 32\gamma n^{-1/3}\log n$ with $\gamma\geq 1$,
then the only local minimum of the energy function $E(\btheta)$ is attained at
\[
\btheta = \btheta_0, 
\]
with probability at least 
\[
1 - 4\exp\left(n \left( \log (100n^{\frac{1}{3}}) - 2\gamma \log n\right)\right) - 10n^{-\gamma+1}.
\]
\end{theorem}
This theorem says for an Erd\H{o}s-R\'enyi graph $G(n,p)$ with $p = \Omega(n^{-1/3}\log n)$, i.e., each node is of degree about $\Omega(n^{2/3}\log n)$, the optimization landscape has no spurious local minima with high probability. In other words, for the homogeneous Kuramoto model whose underlying network is an Erd\H{o}s-R\'enyi graph, all the oscillators will synchronize and the basin of attraction is the whole phase space minus a set of measure zero with high probability. Note that our analysis does not rely on $n\rightarrow \infty$ which thus provides a non-asymptotic characterization of the landscape unlike the existing works~\cite{CMM18,Ichi04,LLYMG16,MT15} on the continuum limit of Kuramoto model. 

As pointed out previously, the energy function can also be derived from the least squares loss for matrix completion restricted to the torus. Without this additional constraint,~\cite{GeLM16} shows that the loss function has no spurious local minima with high probability if $p = \Omega(n^{-1}\log n)$. Our theorem loses a factor of $n^{2/3}$ and requires $p = {\cal O}(n^{-1/3}\log n)$ mainly due to the constraint of entrywise unit modulus.
However, the numerical simulations indicate that the loss of $n^{2/3}$ factor may well be an artifact of proof, as shown in Section~\ref{s:numerics}. We will point out in the proof where our analysis leads to this potentially suboptimal bound. As a result, we also propose the second open problem:
\begin{open}\label{open:optimal}
Is it true that no spurious local minimum exists with high probability if $p = \Omega(n^{-1}\log n)$?
\end{open}
Note that for an Erd\H{o}s-R\'enyi graph, the graph is connected with high probability if $p >n^{-1}\log n.$ Therefore, our conjecture claims as long as the Erd\H{o}s-R\'enyi graph is connected, then the corresponding optimization landscape should be benign with high probability.

\vskip0.2cm

Besides the two aforementioned open problems, our work also motivates a series of mathematical problems, the solutions to which would significantly contribute to the development of nonconvex optimization as well as the understanding of the synchronization phenomena of coupled oscillators on complex networks. While our work concerns the homogeneous Kuramoto model, many mysteries behind the inhomogeneous Kuramoto models remain to be solved where the natural frequencies $\{\omega_i\}_{i=1}^n$ in~\eqref{eq:kuramoto} are nonzero. It is also interesting to look into the Kuramoto model over other random graphs and the graphs with special structures with tools from the recent progress in the nonconvex optimization. The study of the optimization landscape under the semi-random model, cf~\cite{CG18}, is particularly relevant to the robustness of algorithms against the monotone adversary. Other future directions include finding out what kinds of graph properties will determine the optimization landscape of the energy function. As discussed already, the second smallest eigenvalue of the associated graph Laplacian may not be a perfect candidate because the spurious local minima may still exist even if the graph is highly connected. Thus it is especially appealing to find a single quantity associated with the underlying network which is able to characterize the optimization landscape.

\section{Numerics}\label{s:numerics}

In this section, we present some experimental results on the landscape of $E(\btheta)$ with Erd\H{o}s-R\'enyi graph as the underlying network. We consider systems with small and large size, i.e.,\ $n = 5,10,\ldots,100$ and $n = 100,150,\ldots,1500$ respectively. For each system, we run gradient descent on the energy function $E(\btheta)$ with random initialization which is equivalent to a discrete approximation of the Cauchy problem in~\eqref{def:cauchy}. For the systems with smaller size, we vary $p$ between 0 and 1, and for each pair of $(n,p)$ we take 50 independent random instances; for the systems with larger size, we vary $p$ between 0 and $4n^{-1}\log n$ for each $n$ and also perform 50 experiments for each set of parameters. The gradient descent method is implemented with a fixed step size of 0.005 and the maximum number of iterations of 1000 per round; in each round the local search stops when $\|\nabla E(\btheta)\|\leq 10^{-8}$. The shade of grey in the square represents the fraction of instances for which gradient descent reaches its global minimum. The red line represents the phase transition threshold we conjecture in Open Problem~\ref{open:optimal}, i.e.,\ $p = n^{-1}\log n$.

In Figure~\ref{figure:phase1} we observe that the gradient descent with random initialization finds global minimum with high probability as soon as $p$ becomes slightly greater than $n^{-1}\log n$, i.e., the graph is connected.  Figure~\ref{figure:phase2} delivers a similar message that the gradient descent does not converge to global minimum with high probability when $p =n^{-1} \log n$. On the other hand, when $p$ is slightly larger than $n^{-1}\log n$, i.e., $p = 2n^{-1}\log n$, the iterates from gradient descent converge to $\btheta_0$ successfully with high probability. 

Thus the numerical experiments provide some evidences for the Open Problem~\ref{open:optimal} that the phase transition might exist at $p = \Omega(n^{-1}\log n)$, i.e., no spurious local minima exist if $p \geq cn^{-1}\log n$ with high probability. It is important to note that the global convergence of plain gradient descent with random initialization does not necessarily rule out the possible existence of spurious minima. Therefore, the experiments serve to support our belief that a near-optimal bound of $p$ may be obtained by using new techniques.


\begin{figure}[h!]
\centering
\begin{minipage}{0.62\textwidth}
\includegraphics[height=80mm]{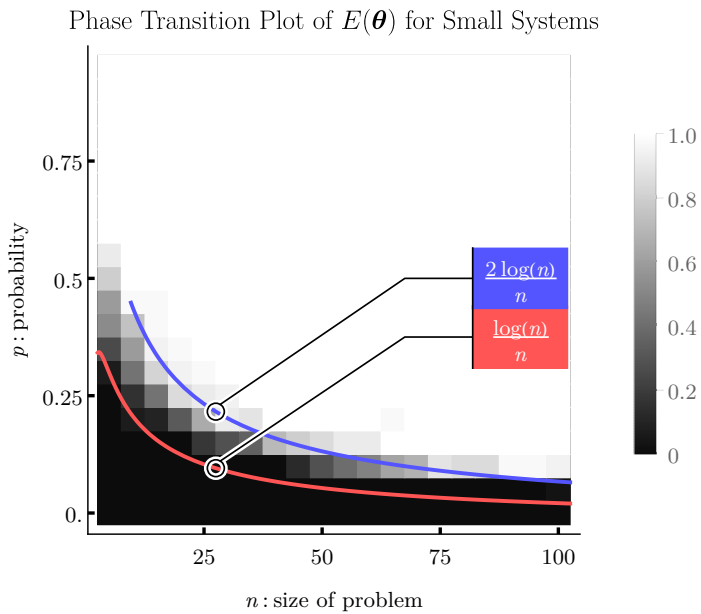}
\end{minipage}
\caption{Phase transition plot on $E(\btheta)$ with smaller systems. 
Here $n$ is chosen between $5$ and 100 and $p$ varies from 0 to 1.
For each pair of $(n,p)$, we run 50 independent experiments. The shade of grey in each square denotes the fraction of trials for which gradient descent reaches its global minimum; the square is lighter if there are more successes. The red line represents the phase transition threshold we conjecture in Open Problem~\ref{open:optimal}, i.e.,\ $p = n^{-1}\log n$; the blue line represents $p = 2n^{-1}\log n$ which achieves high probability of successful recovery.}
\label{figure:phase1}
\end{figure}
\begin{figure}[h!]
\centering
\begin{minipage}{0.62\textwidth}
\includegraphics[height=80mm]{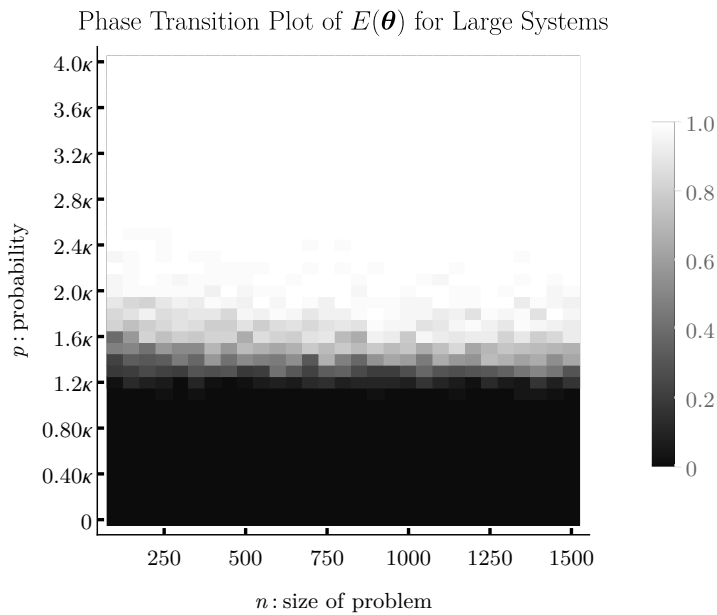}
\end{minipage}
\caption{Phase transition plot on $E(\btheta)$ with larger systems. We let $n$ range from 100 to 1500 and $p =c \kappa$ where $\kappa =\frac{\log n}{n}$ and $0\leq c\leq 4$. For each pair of $(n,p)$, we run 50 independent trials and compute the fraction of successful trials among all 50 instances. The black region stands for failure and the white region means success. A clear phase transition occurs around $p=\frac{1.5\log n}{n}$  and the probability of success is overwhelmingly high when $p \geq \frac{2\log n}{n}$.}
\label{figure:phase2}
\end{figure}

\section{Proofs}\label{s:proofs}
The proofs of main theorems include a key ingredient, which is summarized into the following proposition. This important proposition is inspired by the work of Taylor~\cite{Taylor12}.

\begin{proposition}\label{prop:key}
Let $\btheta = [\theta_1, \cdots, \theta_n]^{\top}$ be a local minimizer of $E(\btheta)$ with $\BA$ as the adjacency matrix of a connected graph. Suppose there exists some $r = \|r\|e^{\mi\theta_r}\in\mathbb{C}$, such that all $\{\theta_i\}_{i=1}^n$ satisfy
\[
|\sin(\theta_i - \theta_r)| < \frac{1}{\sqrt{2}}, \quad\forall~1\leq i\leq n,
\]
then $\btheta = \btheta_0$ is the only local minimizer of $E(\btheta)$ modulo a global translation.
\end{proposition}
\begin{remark}
In the proof of Theorem~\ref{thm:main1} and~\ref{thm:main2}, we pick 
\[
r(\btheta) : = \sum_{j=1}^ne^{\mi\theta_j}
\]
where $n^{-1}r(\btheta)$ is called the order parameter introduced by Kuramoto in~\cite{Kura75}. The absolute value of the order parameter measures how much the current state is correlated with the global minimum. 
\end{remark}
\begin{proof}
Since $\btheta$ is a local minimizer, it must satisfy the first and second necessary condition~\eqref{cond:1st} and~\eqref{cond:2nd} respectively.

The assumption implies that
\[
\{\theta_i\}_{i=1}^n \subseteq \underbrace{\left\{ \theta : - \frac{\pi}{4} < \theta - \theta_r < \frac{\pi}{4}\right\}}_{\Gamma}\bigcup \underbrace{\left\{\theta : \frac{3\pi}{4} < \theta - \theta_r < \frac{5\pi}{4} \right\}}_{\Gamma^c}
\]
which is the union of two disjoint quadrants. 

We claim that all $\{\theta_i\}_{i=1}^n$ must be inside one of these two quadrants and we will prove it by contradiction.
Suppose both quadrants contain at least one element in $\{\theta_i\}_{i=1}^n$.
Since the graph is connected, the induced partitions $\Gamma$ and $\Gamma^c$ yield
\[
\sum_{i\in \Gamma}\sum_{j\in \Gamma^c}a_{ij} \cos(\theta_i - \theta_j) < 0
\]
which follows from $\cos(\theta_i - \theta_j) < 0$ for all $\theta_i \in\Gamma$ and $\theta_j\in\Gamma^c$ and at least one element $a_{ij}$ is nonzero for $i\in\Gamma$ and $j\in\Gamma^c$ due to the graph connectivity.
However, this contradicts the fact that $\ddiag(\BA\circ \BQ\BQ^{\top}) - \BA\circ \BQ\BQ^{\top} \succeq 0$. 

By letting $\bw\in \RR^n$ whose $i$-th entry $w_i$ satisfies
\[
w_i = 
\begin{cases}
1, & \text{if } \theta_i\in\Gamma, \\
-1, & \text{if }\theta_i\in\Gamma^c,
\end{cases}
\]
there holds
\begin{align*}
\sum_{i\in \Gamma}\sum_{j\in \Gamma^c}a_{ij} \cos(\theta_i - \theta_j) & = \frac{1}{4} \sum_{i\in\Gamma}\sum_{j\in\Gamma^c} a_{ij}\cos(\theta_i - \theta_j) (w_i - w_j)^2\\
& = \frac{1}{4}\bw^{\top} (\ddiag(\BA\circ \BQ\BQ^{\top})-\BA\circ \BQ\BQ^{\top}) \bw \geq 0
\end{align*}
which leads to a contradiction. 

Therefore, all $\{\theta_i\}_{i=1}^n$ must stay in the interior of a single quadrant. Now we show that if this is the case, the only critical point of $E(\btheta)$ within this quadrant must be $\btheta = \bzero$.

Define an auxiliary function as
\[
E(t) \coloneqq \frac{1}{2}\sum_{1\leq i,j\leq n}a_{ij}(1-\cos( (\theta_i - \theta_j)t ))
\]
where $t\in[0,1].$
Then
\[
E'(t) = \frac{1}{2} \sum_{1\leq i,j\leq n} a_{ij}\cdot (\theta_i - \theta_j)\sin((\theta_i - \theta_j)t) \geq \frac{t}{\pi} \sum_{1\leq i,j\leq n}a_{ij}(\theta_i - \theta_j)^2
\]
which follows from $x\sin(x) \geq 2\pi^{-1}x^2$ for $-\pi/2\leq x\leq \pi/2$ and $\theta_i - \theta_j \in (-\pi/2,\pi/2).$

From spectral graph theory~\cite[Chapter 1]{Chung97}, we know the quadratic form $\sum_{i,j}a_{ij}(\theta_i - \theta_j)^2$ must be strictly positive for a connect graph unless $\theta_i =\theta_j$ for $i\neq j$. In other words, if $\btheta \neq \bzero$ (modulo a global translation), $E'(t)  > 0$ implies that $\btheta$ cannot be a critical point of $E(\btheta)$ since the energy function strictly decreases along the straight line pointing from $\btheta$ to the origin. Thus $\btheta$ is not a local minimizer unless $\btheta = \bzero$. This guarantees that $\btheta = \bzero$ holds and it is the only local minimizer within the quadrant. 
\end{proof}

\subsection{Proof of Theorem~\ref{thm:main1}}

\begin{proof}[{\bf Proof of Theorem~\ref{thm:main1}}]
Consider a graph with adjacency matrix $\BA$ and each vertex has degree of at least $\mu (n-1)$, i.e., 
\[
\min_{1\leq i\leq n} \sum_{j:j\neq i} a_{ij} \geq \mu (n-1).
\]
Here we assume the diagonal entries of $\BA$ are zero, i.e., $a_{ii}=0$, since they do not contribute to $E(\btheta).$

Proposition~\ref{prop:key} indicates that it suffices to show that $\{\theta_i\}_{i=1}^n$ be within two disjoint quadrants for any given local minimizer $\btheta$ under $\mu > (3-\sqrt{2})/2.$
For any given local minimizer $\btheta$, we let $\BQ\in\RR^{n\times 2}$ be of the form in~\eqref{def:Q} and we have $(\BQ\BQ^{\top})_{ij} = \cos(\theta_i - \theta_j).$ The second order necessary condition in~\eqref{eq:2nd} implies that
\[
\lag \diag(\BA\BQ\BQ^{\top}) - \BA\circ \BQ\BQ^{\top}, \BQ\BQ^{\top}\rag = \lag \BA,\BQ\BQ^{\top}\rag - \lag \BA,\BQ\BQ^{\top}\circ \BQ\BQ^{\top}\rag \geq 0
\]
where $(\BQ\BQ^{\top})_{ii} = 1, 1\leq i\leq n$ and $\BQ\BQ^{\top}$ is positive semidefinite. 
The inequality above is equivalent to
\begin{align}
\lag \BJ_n , \BQ\BQ^{\top}\rag & \geq \lag \BJ_n , \BQ\BQ^{\top}\circ \BQ\BQ^{\top}\rag + \lag \BJ_n - \BA, \BQ\BQ^{\top} - \BQ\BQ^{\top}\circ \BQ\BQ^{\top}\rag \nonumber \\
&\geq \lag \BJ_n , \BQ\BQ^{\top}\circ \BQ\BQ^{\top}\rag + \lag \BJ_n - \I_n - \BA, \BQ\BQ^{\top} - \BQ\BQ^{\top}\circ \BQ\BQ^{\top}\rag. \label{eq:JQQ}
\end{align}
Our first goal is to get a lower bound of $\lag \BJ_n, \BQ\BQ^{\top}\rag$ by estimating each term on the right hand of~\eqref{eq:JQQ}.

Note that $\BQ\BQ^{\top}$ is rank-2 and its trace is $n$. Thus its Frobenius norm has a lower bound as
\begin{equation}\label{eq:JQQsq}
\lag \BJ_n, \BQ\BQ^{\top} \circ \BQ\BQ^{\top}\rag = \|\BQ\BQ^{\top}\|_F^2 = \lambda_1^2 + \lambda_2^2 \geq \frac{(\lambda_1 + \lambda_2)^2}{2} = \frac{1}{2}(\Tr(\BQ\BQ^{\top}))^2 = \frac{n^2}{2}
\end{equation}
where $\lambda_1$ and $\lambda_2$ are the eigenvalues of $\BQ\BQ^{\top}$ and $\Tr(\BQ\BQ^{\top}) = n.$

Since each row sum of $\BA$ is at least $\mu(n-1)$, we know that $\BJ_n - \I_n - \BA$ has at most $(1-\mu)(n-1)n$ nonzero entries. Moreover, the absolute values of each entry in $\BQ\BQ^{\top}$ and $\BQ\BQ^{\top}\circ \BQ\BQ^{\top}$ are bounded by 1. Hence,
\begin{align}
|\lag \BJ_n -\I_n - \BA, \BQ\BQ^{\top} - \BQ\BQ^{\top}\circ \BQ\BQ^{\top}\rag| 
& \leq \sum_{i\neq j}(1-a_{ij}) | \cos(\theta_i - \theta_j) - \cos^2(\theta_i - \theta_j)| \nonumber \\
& \leq 2\sum_{i\neq j}(1 - a_{ij}) \leq 2(1-\mu)(n-1)n. \label{eq:JQQup}
\end{align}

Combining~\eqref{eq:JQQsq} and~\eqref{eq:JQQup} with~\eqref{eq:JQQ}, we have a lower bound for the correlation between $\BQ$ and the constant vector $\bm 1_n$ as follows
\begin{align*}
\lag\BJ_n , \BQ\BQ^{\top}\rag   \geq \frac{n^2}{2} - 2(1-\mu)(n-1)n = \left( 2\mu - \frac{3}{2}\right)n^2 + 2(1-\mu)n.
\end{align*}

Now we consider the rescaled order parameter $r$ as 
\[
r = \sum_{j=1}^n e^{\mi \theta_j} = \|r\| e^{\mi \theta_r}
\]
and its $\ell_2$-norm is bounded below by
\begin{align}
\|r\|^2 & = \left(\sum_{j=1}^n \cos \theta_j\right)^2 + \left(\sum_{j=1}^n \sin \theta_j\right)^2 = \sum_{1\leq i,j\leq n} \left( \cos\theta_i \cos\theta_j +  \sin\theta_i \sin\theta_j\right) \nonumber \\
& = \lag\BJ_n , \BQ\BQ^{\top}\rag  \geq \left( 2\mu - \frac{3}{2}\right)n^2 + 2(1-\mu)n. \label{eq:rlow}
\end{align}


\vskip0.5cm
For every $i\in[n]$, there holds
\[
\|r\| e^{-\mi(\theta_i - \theta_r)} = r e^{-\mi\theta_i} = \sum_{j=1}^n e^{-\mi(\theta_i - \theta_j)}.
\]
By taking the imaginary parts, we have
\begin{equation}\label{eq:r-sin}
\|r\|\cdot \sin(\theta_i - \theta_r) = \sum_{j=1}^n \sin(\theta_i - \theta_j) = \sum_{j=1}^n (1 - a_{ij}) \sin(\theta_i - \theta_j), \quad \forall 1\leq i\leq n,
\end{equation}
where the second equality follows from the first order necessary condition in~\eqref{cond:1st}, i.e., 
\[
\sum_{j=1}^n a_{ij}\sin(\theta_i - \theta_j) = 0, \quad 1\leq i\leq n.
\]

By combining~\eqref{eq:JQQup} with~\eqref{eq:r-sin}, we have
\[
\|r\| \cdot |\sin(\theta_i - \theta_r)| = \left| \sum_{j=1}^n (1 - a_{ij}) \sin(\theta_i - \theta_j) \right| \leq \left| \sum_{j=1}^n (1 - a_{ij}) \right| \leq  (1-\mu)(n-1).
\]
Note that
$\|r\|^2 = \lag \BJ_n , \BQ\BQ^{\top}\rag \geq  \left( 2\mu - \frac{3}{2}\right)n^2 + 2(1-\mu)n$ in~\eqref{eq:rlow}. Therefore,
\[
|\sin(\theta_i - \theta_r)|^2 \leq \frac{(1-\mu)^2(n-1)^2}{\|r\|^2} \leq \frac{2(1-\mu)^2}{4\mu-3 + 4(1-\mu)n^{-1}} < \frac{1}{2}
\]
for any $1\leq i\leq n$ provided that $(3-\sqrt{2})/2 \leq \mu \leq 1$, which immediately implies
\[
|\sin(\theta_i - \theta_r)| < \frac{1}{\sqrt{2}}, \quad 1\leq i\leq n.
\]
Here $r$ and $\btheta$ satisfy the assumption in Proposition~\ref{prop:key}, and moreover the underlying graph is connected given that $\mu \geq (3-\sqrt{2})/2$. Now applying Proposition~\ref{prop:key} yields the desired result. 
\end{proof}

\subsection{Proof of Theorem~\ref{thm:main2}}
\subsubsection{The road map for proof of Theorem~\ref{thm:main2}}
To prove Theorem~\ref{thm:main2}, besides Proposition~\ref{prop:key}, we will also use the toolbox of random matrix and concentration of measure which are widely used in many signal processing and statistical inference problem. 
Simply speaking, the main strategy follows from finding a lower bound of the rescaled order parameter $\|r\| = \|\sum_{j=1}^n e^{\mi \theta_j}\|$ which is exactly equal to $\|\BQ^{\top}\bone_n\|$; obtaining an upper bound of $|\Imag(re^{-\mi \theta_i})|$ for a local minimizer $\btheta = \{\theta_i\}_{i=1}^n$; and finally applying Proposition~\ref{prop:key} to finish the proof. Note that by applying Theorem 15 in~\cite{BBV16} with a proper transformation, the same performance bound with respect to $p$ can be obtained to ensure a benign optimization landscape of $E(\btheta)$. We provide an alternative proof based on the ``two-disjoint-quadrant" argument (Proposition~\ref{prop:key}) and get a tighter bound of the lower bound of order parameter, which are of independent interests.

From now on, we set 
\begin{equation}\label{def:p}
p := \frac{32\gamma \log n}{n^{1/3}}
\end{equation}
in~\eqref{eq:ER} under the assumptions of Theorem~\ref{thm:main2} and also define
\[
\BDelta := \BA - p\BJ_n
\]
as the centered random adjacency matrix. Here for the simplicity of presentation, we set the diagonal entries of $\BA$ equal to $p$ since they do not appear in $E(\btheta).$

Now we present two supporting lemmata whose proof we defer to Section~\ref{ss:supportlemma} and show how these two results lead to the final proof of Theorem~\ref{thm:main2}.
\begin{lemma}\label{lem:aux}
For $\BA$ be an Erd\H{o}s-R\'enyi random graph $G(n,p)$, we have
\begin{equation}\label{eq:A-EA}
\| \BA - \E(\BA)\| \leq \sqrt{2\gamma np(1-p)\log n}  + \frac{2\gamma \log n}{3} 
\end{equation}
with probability at least $1 - 2n^{-\gamma+1}.$ In particular, there holds
\[
\max_{1\leq i\leq n} \| \be^{\top}_i(\BA - \E(\BA))\| \leq \sqrt{2\gamma np(1-p)\log n}  + \frac{2\gamma \log n}{3}.
\]
Moreover, the infinity norm bound of $\BDelta\bone_n$ obeys
\[
 \|\BDelta\bone_n\|_{\infty} \leq \sqrt{2\gamma np (1-p)\log n}  + \frac{2\gamma \log n}{3}
\]
 with probability at least $1 - 2n^{-\gamma+1}.$

\end{lemma}
The proof of this lemma is very standard; it follows from direct application of matrix Bernstein inequality. The purpose of this lemma is to show how close $\BA$ is to its expectation; moreover the estimation of the upper bound of $|\Imag(re^{-\mi \theta_i})|$ is based on this fact. 
The next lemma demonstrates that with high probability, the rescaled order parameter of a local minimizer, i.e.,  $r = \sum_{j=1}^n e^{\mi \theta_j}$, and the global minimizer $\bone_n$ are highly correlated.
\begin{lemma}\label{lem:Q1_low}
With probability at least 
\[
1 - 4\exp\left(n \left( \log (100n^{1/3}) - 2\gamma\log n\right)\right), \quad \gamma\geq 1,
\]
there holds
\[
\|\BQ^{\top}\bone_n\|^2 \geq n^2 (1 - 3n^{-\frac{1}{3}})
\]
uniformly for all $\BQ\in\RR^{n\times 2}$ satisfying the first and second necessary conditions.
\end{lemma}

\begin{proof}[\bf Proof of Theorem~\ref{thm:main2}]
Let $r(\btheta) := \sum_{j=1}^n e^{\mi \theta_j}$ where $\{\theta_j\}_{j=1}^n$ is a local minimizer and for simplicity, we use ``$r$" to represent $r(\btheta)$.
By using the first order necessary condition $\sum_{j=1}^n a_{ij}\sin(\theta_i - \theta_j) = 0$ and Cauchy-Schwarz inequality, we have 
\begin{align*} 
p |\Imag(re^{-\mi\theta_i})|& =  p\left| \sum_{j=1}^n \sin(\theta_i - \theta_j)\right|  \\
& = \left| \sum_{j=1}^n (a_{ij} - p) \sin(\theta_i - \theta_j) \right| \\
& \leq  |\sin\theta_i| \cdot \left| \sum_{j=1}^n (a_{ij} - p)\cos \theta_j\right| +  |\cos\theta_i| \cdot \left| \sum_{j=1}^n (a_{ij} - p)\sin \theta_j\right| \\
& \leq \sqrt{ \left(\sum_{j=1}^n (a_{ij}-p) \cos \theta_j \right)^2  + \left(\sum_{j=1}^n (a_{ij}-p) \sin \theta_j \right)^2 } \\
& = \| \be_i^{\top}(\BA-p\BJ_n)\BQ\| = \|\be_i^{\top}\BDelta \BQ\|.
\end{align*}

Now we give an upper bound of $\| \be_i^{\top}(\BA-p\BJ_n)\BQ\|$:
\begin{align*}
\| \be_i^{\top}\BDelta \BQ\| & \leq \frac{1}{n} | \be_i^{\top} \BDelta \bone_n | \cdot \| \BQ^{\top}\bone_n \|  + \left\|\be_i^{\top}\BDelta \left(\I_n - \frac{\bone_n\bone_n^{\top}}{n}\right)\BQ\right\| \\
& \leq \frac{1}{n}\|\BDelta \bone_n\|_{\infty} \cdot \|\BQ^{\top}\bone_n\| + \|\be_i^{\top}\BDelta \| \cdot \left\| \left(\I_n - \frac{\bone_n\bone_n^{\top}}{n}\right)\BQ\right\|_F.
\end{align*}
In fact, the suboptimal bound of our theorem results from the estimation of the second term above. The possible solution may require a sharper upper bound of the correlation between the local minimizers $\BQ$ and each row of $\BDelta$.

It is straightforward to get an upper bound for each term in the expression above, 
\begin{align*}
& \|\BDelta \bone_n\|_{\infty} \leq \sqrt{2\gamma np(1-p)\log n} + \frac{2\gamma \log n}{3}, \\
& \|\be_i^{\top}\BDelta \| \leq \sqrt{2\gamma np (1-p)\log n} + \frac{2\gamma \log n}{3}, \\
& \left\| \left(\I_n - \frac{\bone_n\bone_n^{\top}}{n}\right)\BQ\right\|_F^2 = n - \frac{1}{n}\|\BQ^{\top}\bone_n\|^2 \leq n - n\left(1 - 3n^{-\frac{1}{3}}\right)  \leq 3n^{\frac{2}{3}}, \\
& \|\BQ^{\top}\bone_n\| \leq n,
\end{align*}
where the first two inequalities are based on Lemma~\ref{lem:aux} and the last two are implied by Lemma~\ref{lem:Q1_low}.

Therefore, each $\|\be_i^{\top}\BDelta \BQ\|$ is uniformly bounded by
\begin{align*}
\| \be_i^{\top}\BDelta \BQ\| & \leq \left( 1 + \sqrt{3}n^{\frac{1}{3}}\right)\cdot \left( \sqrt{2\gamma n p(1-p)\log n} + \frac{2\gamma \log n}{3} \right) 
\end{align*}
for all $\BQ$.

For reasonably large $n$, we consider $p^{-1}\|\be_i^{\top}\BDelta \BQ\|$ and it satisfies
\begin{align}\label{ineq:perturb}
|\Imag(re^{-\mi\theta_i})| & \leq \frac{\|\be_i^{\top}\BDelta \BQ\|}{p} \leq 
2n^{\frac{1}{3}} \left( \sqrt{ \frac{ 2\gamma n (1-p)\log n}{p}} + \frac{2\gamma \log n}{3p} \right) \nonumber \\
& =  \frac{n}{2}  \left( \sqrt{ 1-p}  + \frac{n^{-\frac{1}{3}}}{12} \right) \nonumber \\
 & < n\sqrt{\frac{ 1-3n^{-\frac{1}{3}}}{2}}
\end{align}
when $p = \frac{32\gamma \log n}{n^{1/3}} > \frac{3}{n^{1/3}}$ for $\gamma\geq 1$. On the other hand, we have
\begin{align}\label{ineq:correlation}
\|r\|^2 = \|\BQ^{\top}\bone_n\|^2 \geq n^2(1 - 3n^{-\frac{1}{3}}),
\end{align}
where the lower bound is given by Lemma~\ref{lem:Q1_low}.

For every $i\in [n]$ we consider an upper bound of $|\sin(\theta_i - \theta_r)|$ by using
\[
\left| \sum_{j=1}^n \sin(\theta_i - \theta_j)\right| = \left| \Imag(re^{-\mi\theta_i}) \right| = \|r\| \cdot |\sin(\theta_i - \theta_r)| \leq \frac{\|\be_i^{\top}\BDelta \BQ\|}{p}
\]
where $r = \sum_{j=1}^n e^{\mi\theta_j} = \|r\|e^{\mi\theta_r}$.
Combining~\eqref{ineq:perturb} with~\eqref{ineq:correlation} yields
\begin{equation}\label{eq:angle}
 |\Imag(re^{-\mi\theta_i})| = \|r\|\cdot |\sin(\theta_i -\theta_r)| \leq \frac{ \|\be_i^{\top}\BDelta \BQ\|}{p} < n\sqrt{\frac{ 1-3n^{-\frac{1}{3}}}{2}} \leq \frac{\|\BQ^{\top}\bone_n\|}{\sqrt{2}} = \frac{\|r\|}{\sqrt{2}}.
\end{equation}
Therefore, we have
\[
  |\sin(\theta_i - \theta_r)| < \frac{1}{\sqrt{2}},\quad \forall\, 1\leq i\leq n.
\]
On the other hand, we assume $p = 32\gamma n^{-\frac{1}{3}}\log n \gg n^{-1}\log n$ and thus by Theorem 5.8 in~\cite[Chapter 5]{Hof16}, the Erd\H{o}s-R\'enyi random graph is connected with overwhelmingly high probability. Now all the assumptions of Proposition~\ref{prop:key} are fulfilled and we have $\btheta = \bzero$ as the only local minimizer.
\end{proof}

\subsubsection{Proof of the supporting lemmata in Theorem~\ref{thm:main2}}
\label{ss:supportlemma}
\begin{theorem}[\bf Matrix Bernstein, Theorem 1.4 in~\cite{Tropp12}]
Consider a finite sequence $\{\BZ_k\}$ of independent random, self-adjoint matrices with dimension $n\times n$. Assume that each random matrix satisfies
\[
\E(\BZ_k) = 0, \qquad \|\BZ_k\| \leq R.
\]
Then for all $t \geq 0$
\[
\Pr\left( \left\| \sum_{k}\BZ_k\right\| \geq t\right) \leq n\cdot \exp\left( -\frac{t^2/2}{\sigma^2 + Rt/3}\right)
\]
where 
$\sigma^2 = \|\sum_k \E(\BZ_k^2)\|.$ In other words,
\[
\left\| \sum_{k}\BZ_k\right\| \leq \frac{2\gamma R \log n}{3} + \sqrt{2\gamma\sigma^2\log n}
\]
holds with probability at least $1 - 2n^{-\gamma+1}.$

\end{theorem}

\begin{proof}[{\bf Proof of Lemma~\ref{lem:aux}}]
The difference between $\BA$ and $\E(\BA)$ is a sum of rank-1 matrices:
\[
\BA - \E(\BA) = \sum_{i<j} (a_{ij} - p)(\BE_{ij} + \BE_{ji})
\]
where $\E(a_{ij}) = p$ and $\BE_{ij} = \be_i\be_j^{\top}$.  Note that each component is bounded by $\max\{p, 1-p\}$ in operator norm since $\| \BE_{ij} + \BE_{ji} \| = 1$.
The variance is
\[
\E (\BA - \E(\BA))^2 = p(1-p)\sum_{i<j} (\BE_{ii} + \BE_{jj}) \preceq np(1-p) \I_n
\]
which follows from the independence among all entries $\{a_{ij}\}.$ Applying matrix Bernstein inequality, we have
\[
\Pr( \| \BA - \E(\BA) \| \geq t ) \leq 2n\cdot \exp\left( -\frac{t^2/2}{np(1-p) + t/3}\right) \leq 2n^{-\gamma + 1} 
 \]
 with $t = \sqrt{2\gamma np(1-p)\log n} + 2\gamma \log n/3$. 
 The upper bound of $\|\be_i^{\top}(\BA - \E(\BA))\|$ follows directly from $\|\be_i^{\top}(\BA - \E(\BA))\| \leq \|\BA - \E(\BA)\|$.
 
 For the infinity norm bound of $\BDelta \bone_n$, we apply Bernstein inequality to $(\BDelta \bone_n)_i$ for each fixed $i$ and then take the union bound over $1\leq i\leq n$, and then have
 \[
\Pr\left( \|\BDelta\bone_n\|_{\infty} \geq t \right) = \Pr\left( \max_{1\leq i\leq n}\left|\sum_{j=1}^n (a_{ij} - p) \right| \geq t \right) \leq 2n\exp\left(-\frac{t^2/2}{np(1-p) + t/3}\right) \leq 2n^{-\gamma+1}
\]
where $t = \sqrt{2\gamma np(1-p)\log n} + 2\gamma \log n/3.$

\end{proof}

In order to prove Lemma~\ref{lem:Q1_low}, we first introduce the restricted isometry property of $\BA - p\BJ_n$ as follows.

\begin{lemma}[{\bf Restricted isometry property on manifold}]\label{lem:RIP}
Let $\BA$ be an Erd\H{o}s-R\'enyi graph $G(n,p)$. Then the following restricted isometry property holds
\begin{align}\label{eq:RIP}
\begin{split}
| \lag \BA - p\BJ_n, \BQ\BQ^{\top}\circ \BQ\BQ^{\top}\rag | & \leq \delta p\|\BQ\BQ^{\top}\|_F^2, \\
| \lag \BA - p\BJ_n, \BQ\BQ^{\top}\rag| & \leq \delta p\|\BQ\BQ^{\top}\|_F^2, 
\end{split}
\end{align}
uniformly for all $\BQ$ with probability at least 
\[
1 -4\exp\left( n\log \left(\frac{100}{\delta}\right) -\frac{n^2\delta^2p }{64(1-p) + 6\delta}\right)
\]
where $0<\delta<1.$
\end{lemma}
\begin{proof}

Let $\BX := \BQ\BQ^{\top}$ and $X_{ij} = \cos(\theta_i - \theta_j).$ We expect $\lag \BA, \BX \circ \BX\rag$ and $\lag \BA, \BX\rag$ will not deviate  too much from their expectations~\emph{uniformly} for all $\BX= \BQ\BQ^{\top}$ where
\[
\lag \BA, \BX\circ \BX\rag = \sum_{i,j} a_{ij} \cos^2(\theta_i - \theta_j), \quad \lag \BA, \BX\rag = \sum_{i,j} a_{ij} \cos(\theta_i - \theta_j)
\]
are subgaussian random variables. It suffices to only look at $\lag \BA, \BX\circ \BX\rag$ due to their similarities. 

Note that the centered variable satisfies
\[
\lag \BA, \BX\circ \BX \rag - p\|\BX\|_F^2 = \sum_{i \neq j}(a_{ij} - p) \cos^2(\theta_i - \theta_j)
\]
which follows from $\E(a_{ij}) = p.$
Each term $(a_{ij} - p)\cos^2(\theta_i - \theta_j)$ is contained in $[-p, 1-p].$ Moreover, the variance of $\sum_{i,j}(a_{ij} - p)\cos^2(\theta_i - \theta_j)$ satisfies
\[
\sigma^2 = \Var\left( 2\sum_{i<j} (a_{ij}-p)\cos^2(\theta_i - \theta_j)\right) = 4p(1-p)\sum_{i<j}\cos^4(\theta_i - \theta_j) \leq 2n^2p(1-p).
\]

By Bernstein's inequality, we have
\begin{align}
\Pr\left( \left| \lag \BA, \BX\circ \BX \rag  - p\|\BX\|_F^2  \right|  \geq t  \right) 
& \leq 2\exp\left( -\frac{t^2/2}{2n^2p(1-p) + t/3}\right) \nonumber \\
& \leq 2\exp\left( -\frac{n^2\delta^2 p}{64(1-p) + 6\delta } \right) \label{eq:AXX}
\end{align}
where we use
\[
t = \frac{\delta p\|\BX\|_F^2}{2}, \quad \frac{n^2}{2}\leq \|\BX\|_F^2 \leq n^2.
\]

Thus there holds
\begin{equation}\label{eq:AXX1}
\left| \lag \BA, \BX\circ \BX\rag - p \|\BX\|_F^2 \right| \leq \frac{\delta p}{2}  \|\BX\|_F^2
\end{equation}
with probability at least 
\[
1 -2\exp\left( -\frac{n^2\delta^2 p}{64(1-p) + 6\delta } \right).
\]
\vskip0.25cm

So far we have established concentration inequality for a fixed set of $\{\theta_i\}_{i=1}^n.$ Now we need to extend the inequality~\eqref{eq:AXX1} to arbitrary $\theta\in [0,2\pi]^{n}$ uniformly.

Define $m \coloneqq \lfloor \eps^{-1}+1\rfloor$ and we construct an $\eps$-net over $[0,2\pi]^n$ as
\[
{\cal S}_{n,m} := \left\{ \widetilde{\btheta} = \{\widetilde{\theta}_i\}_{i=1}^n\in[0, 2\pi]^n: \widetilde{\theta}_i  = \frac{2l\pi}{m} \text{ for some }l:1\leq l\leq m\right\}
\]
with $|{\cal S}_{n,m}| = m^n$. 
The formula~\eqref{eq:AXX} holds for all $\widetilde{\btheta}$ on the net ${\cal S}_{n,m}$ with probability  at least 
\[
1 - 2|{\cal S}_{n,m}|\cdot \exp\left( -\frac{n^2\delta^2p}{64(1-p) + 6\delta}\right) = 1 -2 \exp\left( n\log m -\frac{n^2\delta^2p}{64(1-p) + 6\delta}\right)
\]
by taking the union bound for all $\widetilde{\btheta}\in {\cal S}_{n,m}.$

For arbitrary $\{\theta_i\}_{i=1}^n$, there always exists a point $\widetilde{\btheta}$ on the net such that $\|\btheta - \widetilde{\btheta}\|_{\infty} \leq \pi/m$, i.e., $|\theta_i - \widetilde{\theta}_{i}| \leq \pi/m$ for $1\leq i\leq n$. By letting $f(\btheta) := \lag \BA- p\BJ_n, \BQ\BQ^{\top}\circ \BQ\BQ^{\top}\rag$, then $f(\btheta)$ is a Lipschitz continuous function and satisfies 
\begin{align}
|f(\btheta) - f(\widetilde{\btheta})| & 
= \left| \sum_{i\neq j} (a_{ij} - p) (\cos^2(\theta_i - \theta_j) - \cos^2(\widetilde{\theta}_{i} - \widetilde{\theta}_j)) \right| \nonumber \\
& \leq 2\sum_{i\neq j}|a_{ij} - p| \cdot  |(\theta_i - \theta_j) - (\widetilde{\theta}_{i} - \widetilde{\theta}_j)| \nonumber \\
& \leq 4\|\btheta - \widetilde{\btheta}\|_{\infty}\cdot \sum_{i\neq j}|a_{ij} - p| \nonumber \\
& \leq 4\|\btheta - \widetilde{\btheta}\|_{\infty}\cdot \sum_{i\neq j}(a_{ij} + p) \nonumber \\
& \leq \frac{4\pi}{m} \left( n\|\BA\| + n^2p\right)  \label{eq:ftheta}
\end{align} 
where $|\theta_i - \widetilde{\theta}_i |\leq \pi/m$ and $\|\BA\| \leq np + \sqrt{2\gamma np(1-p)\log n} + 2\gamma \log n/3.$

Now we consider the difference between $\|\BX\|_F^2$ and $\|\widetilde{\BX}\|_F^2$ where $\widetilde{\BX} = \widetilde{\BQ}\widetilde{\BQ}^{\top}$ and $\widetilde{X}_{ij} = \cos(\widetilde{\theta}_i -\widetilde{\theta}_j )$, and there hold
\begin{align}
\left| \|\BX\|_F^2 - \|\widetilde{\BX}\|_F^2 \right|
& \leq \sum_{i\neq j}| \cos^2(\theta_i - \theta_j) - \cos^2(\widetilde{\theta}_i - \widetilde{\theta}_j)| \nonumber \\
& \leq 2\sum_{i\neq j} |(\theta_i - \theta_j) - (\widetilde{\theta}_i - \widetilde{\theta}_j)| \nonumber \\
& \leq \frac{4\pi n^2}{m}. \label{eq:XX}
\end{align}
Since $|f(\widetilde{\btheta})| \leq \frac{\delta p\|\widetilde{\BX}\|_F^2}{ 2}$, $f(\btheta)$ is bounded by
\begin{align*}
|f(\btheta)| & \leq |f(\widetilde{\btheta})| + |f(\btheta) - f(\widetilde{\btheta})| \\
& \leq \frac{\delta p}{2} \|\widetilde{\BX}\|_F^2 + \frac{4\pi n}{m} (\|\BA\| + np) \\
& \leq \frac{\delta p}{2}\left( \|\BX\|_F^2 +  \frac{4\pi n^2}{m}\right) + \frac{4\pi n}{m} (\|\BA\| + np)  \\
& \leq \frac{\delta p}{2}\|\BX\|_F^2 + \frac{2\pi n^2p}{m}\left(2 + \delta + \sqrt{\frac{2\gamma (1-p)\log n}{np}} + \frac{2\gamma \log n}{3np} \right)
\end{align*}
for any $\btheta$ where~\eqref{eq:A-EA},~\eqref{eq:ftheta}, and~\eqref{eq:XX} are used.
As a result, we have $|f(\btheta)| \leq \delta p\|\BX\|_F^2$ provided that the second term above is bounded by $\frac{n^2\delta p}{4}$, i.e.,
\begin{equation}\label{eq:complex}
m \geq \frac{8\pi}{\delta} \left(2 + \delta + \sqrt{\frac{2\gamma (1-p)\log n}{np}} + \frac{2\gamma \log n}{3np} \right).
\end{equation}
In conclusion, $|\lag \BA - p\BJ_n, \BQ\BQ^{\top} \circ \BQ\BQ^{\top}\rag| \leq \delta p\|\BX\|_F^2$ holds with probability at least 
\[
1 -2 \exp\left( n\log \left(\frac{100}{\delta}\right) -\frac{n^2\delta^2p }{64(1-p) + 6\delta}\right).
\]

Following the similar procedures, we are also able to prove
\[
|\lag \BA - p\BJ_n, \BQ\BQ^{\top}\rag| \leq \delta p\|\BQ\BQ^{\top}\|_F^2
\]
holds for all $\BQ$ uniformly with probability at least $1 - 2 \exp\left( n\log \left(\frac{100}{\delta}\right) - \frac{n^2\delta^2p}{64(1-p) + 6\delta}\right).$ For the completeness, we give a sketch of the proof below. 

Define $g(\btheta) : = \lag \BA-p\BJ_n, \BQ\BQ^{\top}\rag$ and it is straightforward to verify that
\begin{align*}
\Pr\left(|g(\btheta)|\geq \frac{\delta p\|\BQ\BQ^{\top}\|_F^2}{2}\right) = \Pr\left( \left| \lag \BA, \BX \rag  - p\lag \BJ, \BX\rag  \right|  \geq \frac{\delta p\|\BQ\BQ^{\top}\|_F^2}{2}  \right) & \leq 2\exp\left( -\frac{n^2\delta^2 p}{64(1-p) + 6\delta } \right) 
\end{align*}
By using the same $\eps$-net as $\mathcal{S}_{n,m}$ and taking the union bound over all elements in $\mathcal{S}_{n,m}$, we have
\[
\Pr\left( \sup_{\widetilde{\btheta}\in\mathcal{S}_{n,m}}|g(\widetilde{\btheta})| \geq \frac{\delta p\|\BQ\BQ^{\top}\|_F^2}{2}\right) \leq 2\exp\left( n\log m -\frac{n^2\delta^2p}{64(1-p) + 6\delta}\right).
\]
Moreover, there holds
\[
|g(\btheta) - g(\widetilde{\btheta})| \leq \frac{4\pi}{m}(n\|\BA\| + n^2p).
\]
for any $\btheta\in[0,2\pi]^n$ and a point $\widetilde{\btheta}$ on the net $\mathcal{S}_{n,m}$ such that $\|\btheta-\widetilde{\btheta}\|_{\infty}\leq \frac{\pi}{m}$. By using~\eqref{eq:XX} and following the exact same steps below~\eqref{eq:XX}, we can show that 
\[
|g(\btheta)| \leq \delta p\|\BQ\BQ^{\top}\|_F^2
\]
uniformly for all $\btheta\in[0,2\pi]^n$ for $m$ chosen as~\eqref{eq:complex}.
This finishes the proof of restricted isometry property.
\end{proof}

Lemma~\ref{lem:RIP} implies the following corollary directly.
\begin{corollary}[{\bf Restricted isometry property on manifold}]\label{cor:RIP}
Let $\BA$ be an Erd\H{o}s-R\'enyi graph $G(n,p)$, 
\begin{equation}\label{eq:para}
p = \frac{32\gamma\log n}{n^{1/3}}, \qquad \delta = \frac{1}{n^{1/3}}.
\end{equation}
Then the following restricted isometry property holds
\begin{align}\label{eq:RIP2}
\begin{split}
-\delta p\|\BQ\BQ^{\top}\|_F^2 & \leq \lag \BA - p\BJ_n, \BQ\BQ^{\top}\circ \BQ\BQ^{\top}\rag  \leq \delta p\|\BQ\BQ^{\top}\|_F^2, \\
-\delta p\|\BQ\BQ^{\top}\|_F^2 & \leq  \lag \BA - p\BJ_n, \BQ\BQ^{\top}\rag \leq \delta p\|\BQ\BQ^{\top}\|_F^2, 
\end{split}
\end{align}
uniformly for all $\BQ$ with probability at least 
\[
1 - 4\exp\left(n \left( \log (100n^{\frac{1}{3}}) - \frac{\gamma \log n}{2}\right)\right), \quad \gamma\geq 1.
\]

\end{corollary} 

From now on, we set $p$ and $\delta$ as~\eqref{eq:para}.

\begin{proof}[{\bf Proof of Lemma~\ref{lem:Q1_low}}]
The second-order necessary condition implies
\[
\lag \ddiag(\BA\BQ\BQ^{\top}) - \BA\circ \BQ\BQ^{\top}, \BQ\BQ^{\top}\rag  = \lag \BA, \BQ\BQ^{\top}\rag - \lag \BA, \BQ\BQ^{\top}\circ\BQ\BQ^{\top}\rag\geq 0
\]
where $(\BQ\BQ^{\top})_{ii}=1$ for $i\in[n].$

Now we decompose the inequality above into
\[
p\lag \BJ_n, \BQ\BQ^{\top}\rag \geq p\lag \BJ_n, \BQ\BQ^{\top}\circ \BQ\BQ^{\top}\rag + \lag p\BJ_n - \BA, \BQ\BQ^{\top}\rag + \lag \BA - p\BJ_n, \BQ\BQ^{\top}\circ \BQ\BQ^{\top}\rag.
\]
From Lemma~\ref{lem:RIP}, we have a lower bound for the last two terms above, i.e.,
\begin{align*}
\lag p\BJ_n - \BA, \BQ\BQ^{\top}\rag & \geq -\delta p \|\BQ\BQ^{\top}\|_F^2, \\
\lag \BA - p\BJ_n, \BQ\BQ^{\top}\circ \BQ\BQ^{\top}\rag & \geq -\delta p\|\BQ\BQ^{\top}\|_F^2,
 \end{align*}
for all $\BQ.$ Hence the correlation between $\BQ^{\top}$ and $\bone_n$ is bounded below by
\[
\|\BQ^{\top}\bone_n\|^2 \geq (1-2\delta)\|\BQ\BQ^{\top}\|_F^2 \geq \frac{n^2(1-2\delta)}{2}.
\]

Note $\BQ\in\RR^{n\times 2}$ is rank-2, and we let $\BQ = [\bx~ \by]$ with $\bx\perp \by$. This is made possible by multiplying $2\times 2$ orthogonal matrix $\BR$ to the right of $\BQ$ where $\BR$ diagonalizes $\BQ^{\top}\BQ.$ Recall the first order necessary condition~\eqref{cond:1st} can be written into
\[
\BA\bx \circ \by = \BA\by\circ \bx.
\]
By letting $\BA = p\BJ_n + \BDelta$ where $\BDelta = \BA - p\BJ_n$, we have
\[
p \bone_n^{\top}\bx \cdot \by - p \bone_n^{\top}\by\cdot \bx  =  (\BDelta \by)\circ \bx - (\BDelta \bx)\circ \by.
\]
By taking the squared magnitude of the equation above, we arrive at
\[
p^2 (|\bone_n^{\top}\bx|^2 \|\by\|^2 + |\bone_n^{\top}\by|^2 \|\bx\|^2) \leq 2\|\BDelta\|^2(\|\bx\|^2 + \|\by\|^2) =2n\|\BDelta\|^2
\]
where $\bx\perp \by$, $\|\bx\|^2 + \|\by\|^2 = n,$ and 
\[
\| \BDelta \by\circ \bx\| = \|\diag(\bx) \BDelta \by\| \leq \|\BDelta\| \cdot \|\by\|.
\]
Therefore, we have 
\[
|\bone_n^{\top}\bx|^2 \|\by\|^2 + |\bone_n^{\top}\by|^2 \|\bx\|^2 \leq \frac{2n\|\BDelta\|^2}{p^2}.
\]

Without loss of generality, we assume $|\bone_n^{\top}\bx|^2 \geq n^2(1 - 2\delta)/4$ which follows from 
\[
 |\bone_n^{\top}\bx|^2 + |\bone_n^{\top} \by|^2 = \|\BQ^{\top}\bone_n\|^2 \geq \frac{n^2(1-2\delta)}{2}.
\] 
Therefore, $\|\by\|^2$ is bounded by
\[
\|\by\|^2 \leq \frac{2n\|\BDelta\|^2}{p^2|\bone_n^{\top}\bx|^2} \leq \frac{8\|\BDelta\|^2}{np^2(1-2\delta)}
\]
and combined with $\|\bx\|^2 + \|\by\|^2 = n$, we get
\[
\|\bx\|^2\geq n \left( 1 -  \frac{8\|\BDelta\|^2}{n^2p^2(1-2\delta)}\right).
\]
Hence, these local minimizers $\BQ$ would satisfy
\[
\|\BQ\BQ^{\top}\|_F^2 = \|\bx\|^4 + \|\by\|^4  \geq \|\bx\|^4 \geq n^2 \left( 1 -  \frac{8\|\BDelta\|^2}{n^2p^2(1-2\delta)}\right)^2
\]
which yields a refined bound of $\|\BQ^{\top}\bone_n\|^2$, i.e.,
\begin{align*}
\|\BQ^{\top}\bone_n\|^2 & \geq (1 - 2\delta) \|\BQ\BQ^{\top}\|_F^2 \geq (1-2\delta)n^2 \left( 1 -  \frac{8\|\BDelta\|^2}{n^2p^2(1-2\delta)}\right)^2 \\
& \geq (1-2\delta)n^2 \left( 1 - \frac{16\|\BDelta\|^2}{n^2p^2(1-2\delta)} \right) \\
& \geq n^2\left(1  - 2n^{-\frac{1}{3}}- n^{-\frac{2}{3}}\right) \\
& \geq n^2\left(1 - 3n^{-\frac{1}{3}}\right)
\end{align*}
where $\delta = n^{-1/3}$, $p = 32\gamma n^{-1/3} \log n$, and 
\begin{align*}
\frac{4\|\BDelta\|}{np} & \leq 4\left( \sqrt{\frac{ 2\gamma (1-p)\log n}{np}} + \frac{2\gamma \log n}{3np}\right) \\
& \leq n^{-\frac{1}{3}} \sqrt{1-p} + \frac{1}{12}n^{-\frac{2}{3}} \leq n^{-\frac{1}{3}}. 
\end{align*}
\end{proof}

\subsection{Supplementary technical details for the supporting examples}\label{ss:proofex}
\paragraph{The $n$-path graph.}
For the $n$-path graph, its adjacency matrix $\BA$ has a simple form: $a_{ij} = 1$ if $|i-j|=1$ and $0$ otherwise. If we look at the first-order necessary condition, we have
\[
0 = \sum_{j=1}^n a_{ij}\sin (\theta_i - \theta_j) =
\begin{cases} 
\sin(\theta_i - \theta_{i-1}) + \sin(\theta_i - \theta_{i+1}), & \text{if } 2\leq i\leq n-1, \\
\sin(\theta_1 - \theta_2), & \text{if } i=1, \\ 
\sin(\theta_{n-1} - \theta_n), & \text{if } i=n,
\end{cases}
\]
which implies that all critical points must belong to $\{\btheta: \theta_i = 0\text{ or } \pi \mod 2\pi, 1\leq i\leq n\}$ up to a global shift. However, the only global minimum of $E(\btheta)$ w.r.t.\ the $n$-path is $\btheta_0$. Note that for these critical points, the corresponding $\BQ$ is rank-1 and has entries $\pm1$. Now let $\BQ = \bone_{\Gamma} - \bone_{\Gamma^c}$, where $\Gamma = \{i : \theta_i = 0 \}$ and $\Gamma^c = \{i: \theta_i = \pi\}$, and $\bone_{\Gamma}\in\RR^n$ is the corresponding indicator function of $\Gamma$. 
Consider the inner product between the Hessian matrix and $\BQ\BQ^{\top}$, and there holds
\begin{align*}
\lag \ddiag(\BA\BQ\BQ^{\top}) - \BA\circ \BQ\BQ^{\top}, \BQ\BQ^{\top}\rag 
& = \Tr(\BA\BQ\BQ^{\top}) - \lag \BA, \BJ_n\rag \\
& = \bone_{\Gamma}^{\top}\BA\bone_{\Gamma} + \bone_{\Gamma^c}^{\top}\BA\bone_{\Gamma^c} - 2\bone_{\Gamma^c}^{\top}\BA\bone_{\Gamma} - \lag \BA, \BJ_n\rag\\
& = - 4 \bone_{\Gamma^c}^{\top}\BA\bone_{\Gamma} \leq 0
\end{align*}
since $\BQ\BQ^{\top}\circ \BQ\BQ^{\top} = \BJ_n.$ The inequality above is strict unless either $\Gamma$ or $\Gamma^c$ is empty due to the connectivity of the $n$-path. Therefore, for any critical point $\btheta$ where $\btheta\neq \btheta_0$, it cannot be a local minimum because the Hessian matrix is not positive semidefinite.

\paragraph{Wiley-Strogatz-Girvan networks.}
To make the presentation more self-contained, we briefly discuss the network in~\cite{WileySG06} in the language of linear algebra and discrete harmonic analysis, and show the twisted state is a local minimizer when $\mu$ is approximately smaller than $0.68$.
If each node is only connected to its $k$-nearest neighbors, the corresponding adjacency matrix, denoted by $\BA_k$, is a circulant matrix generated by $[0,\underbrace{1,1, \cdots, 1}_{k \text{ one's}}, 0, \cdots, 0, \underbrace{1, \cdots, 1}_{k \text{ one's}}]^{\top}\in\RR^n$, as shown in Figure~\ref{fig:WSG}. In particular, if $k=1,$ then
\[
\BA = 
\begin{bmatrix}
0 & 1 & 0 & \cdots & 0 & 1 \\
1 & 0 & 1 &\cdots & 0 & 0 \\
0 & 1 & 0 & \cdots & 0 & 0 \\
\vdots & \vdots & \vdots &\ddots & \vdots & \vdots \\
0 & 0 & 0 &\cdots & 0 & 1 \\
1 & 0 & 0 &\cdots & 1 & 0 \\
\end{bmatrix}.
\]
An important fact is that all circulant matrices are diagonalizable under Fourier basis, see~\cite[Chapter 4]{GolubV96}.

Let $\omega = e^{2\pi \mi/n}$ be the root of unity and $\bz \coloneqq [1, \omega, \cdots, \omega^{n-1}]^{\top}. $ Since $\bz$ is an eigenvector of $\BA_k$ and $\BA_k$ has real eigenvalues due to its symmetry, we have
\[
\Imag(\diag(\overline{\bz}) \BA_k\bz) = \lambda \Imag(\diag(\overline{\bz})\bz) = 0.
\]
By letting $\bx = \Real(\bz)$ and $\by = \Imag(\bz)$ and substituting $\bz = \bx + \mi\by$ into the equation above, there holds
\[
\Imag(\diag(\bx - \mi \by)\BA_k(\bx + \mi \by)) = \diag(\bx)\BA_k\by - \diag(\by)\BA_k\bx = 0
\]
which indicates that $\widetilde{\btheta} \coloneqq \{2\pi l/n\}_{l=1}^n$ is a critical point of $E(\btheta).$

It suffices to compute the Hessian matrix of $E(\btheta)$ at $\widetilde{\btheta}$. 
Note that $\BQ\BQ^{\top} = \Real( \bz\bz^*)$, and hence $\BA_k\circ \BQ\BQ^{\top} = \Real(\diag(\bz) \BA_k\diag(\overline{\bz}))$ is a symmetric circulant matrix and  so is
\begin{align*}
\BL_k & \coloneqq \ddiag(\BA_k \BQ\BQ^{\top}) - \BA_k \circ \BQ\BQ^{\top}.
\end{align*}

Therefore, $\BL_k$ is also diagonalizable under the Discrete Fourier Transform, and all of its eigenvalues are
\[
\lambda_l(\BL_k) = 2\sum_{j=1}^k \cos\left(\frac{2\pi j}{n}\right) - 2\sum_{j=1}^k \cos\left(\frac{2\pi j}{n}\right) \cos\left(\frac{2\pi (l-1) j}{n}\right), \quad 1\leq l\leq n.
\]
In particular, if $l=2$, the second smallest eigenvalue of the Hessian is
\[
\lambda_2(\BL_k) =  2\sum_{j=1}^k \cos\left(\frac{2\pi j}{n}\right) - 2\sum_{j=1}^k \cos^2\left(\frac{2\pi j}{n}\right) = \frac{2}{n}\lag \BA_k, \BQ\BQ^{\top}-\BQ\BQ^{\top}\circ \BQ\BQ^{\top}\rag.
\]
Therefore, $\widetilde{\btheta}$ is a local minimizer since $E(\btheta)$ if $\lambda_2(\BL_k)>0$ and this is guaranteed by $k \leq 0.34n$ approximately, as suggested in Figure~\ref{fig:lambda2}.

\section*{Acknowledgement}
S.L. and R.X. thank Tianqi Wu for a very fruitful and inspiring discussion as well as the introduction of the reference~\cite{Taylor12}. 

\bibliography{phase_sync}
\bibliographystyle{abbrv}

\end{document}